\documentclass[11pt]{amsart}

\usepackage{graphicx}
\usepackage{amsmath,amssymb,amsthm,hyperref}
\usepackage[capitalise]{cleveref}
\usepackage{fullpage}
\usepackage{enumerate}
\usepackage{enumitem}
\usepackage{pgf,tikz}
\usetikzlibrary{patterns}
\usetikzlibrary{arrows}
\usetikzlibrary{positioning}
\usepackage{thmtools,thm-restate}

\def\COMMENT#1{}
\let\COMMENT=\footnote
\newtheorem{theorem}{Theorem}[section]
\newtheorem{example}[theorem]{Example}
\newtheorem{lemma}[theorem]{Lemma}

\newtheorem{claim}[theorem]{Claim}
\newtheorem{fact}[theorem]{Fact}

\newtheorem{proposition}[theorem]{Proposition}
\newtheorem{remark}[theorem]{Remark}
\newtheorem{observation}[theorem]{Observation}

\def\eps{\varepsilon}

\theoremstyle{definition}
\newtheorem{define}[theorem]{Definition}

\newcommand{\qedclaim}{\hfill$\blacksquare$}
\newenvironment{proofclaim}{\removelastskip\penalty55\medskip\noindent{\it Proof of the claim.}}{\medskip\phantom{.}\hfill\qedclaim}

\title{Ramsey numbers of multiple copies of a graph and the random Ramsey theorem}
\author{Andrea Freschi \and Ryan R. Martin \and Andrew Treglown}
\thanks{AF: HUN-REN Alfr\'ed R\'enyi Institute of Mathematics, Budapest, Hungary. 
Research partially supported by ERC\break 
\indent Advanced Grants ``GeoScape", no. 882971 and ``ERMiD", no. 101054936.
Email: \texttt{freschi.andrea@renyi.hu}.
\\
\indent RRM: Department of Mathematics, Iowa State University, United States of America.
 Research partially supported \indent by a Simons Collaboration Grant for Mathematicians, \#709641, and by a Fulbright Scholarship at the University 
\indent of Birmingham, UK. Special thanks to the University  of Birmingham  and the US-UK Fulbright Commission. \indent Email: \texttt{rymartin@iastate.edu}.
\\
\indent  AT: School of Mathematics, University of Birmingham, United Kingdom. 
Research supported by EPSRC grant \indent UKRI1117.
Email: \texttt{a.c.treglown@bham.ac.uk}. }
\date{\today}

\begin{document}

\begin{abstract}
A well-known result of Burr, Erd\H{o}s and Spencer [Transactions of the American Mathematical Society, 1975] determines the $2$-colour Ramsey number for any sufficiently large collection of vertex-disjoint copies of a fixed graph $H$ without isolated vertices. A focus of this paper is to give analogous results for the corresponding $r$-colour Ramsey problem.
More precisely we determine, up to an \emph{additive constant}, this Ramsey number in the case when $r=3$ and also in the case when the chromatic number of $H$ is at least $r$. In these cases our results depend on a parameter which, roughly speaking, describes the corresponding class of potential extremal colourings of the complete graph.
Our proofs rely on our notion of {\it $(H,r)$-gadgets}, which is crucial to obtain results that are best-possible up to additive constant terms.
We exploit this notion using linear programming and the Poincar\'e--Miranda theorem.

We also determine, up to a linear error term, the corresponding $r$-colour Ramsey number in the case when $H$ is a complete bipartite graph.
Here, the corresponding class of potential extremal examples exhibits a connection with the well-known clique-edge-covering problem.

We also prove random versions of some of our results. In particular,
we prove a random version of the Burr--Erd\H{o}s--Spencer theorem, thereby generalising the \textit{random Ramsey theorem} of R\"odl and Ruci\'nski [Journal of the American Mathematical Society, 1995]. 
Our proofs make use of coloured versions of the (sparse) regularity lemma and the K{\L}R conjecture for random graphs.
\end{abstract}

\maketitle

\section{Introduction}
\subsection{Ramsey numbers of multiple copies of a graph}\label{sec11}
One of the most central topics in combinatorics is \textit{Ramsey theory}: the study of partitions of mathematical objects, and in particular, what structures one can guarantee in such partitions. Ramsey's foundational work~\cite{Ramsey} yields the following result for graphs: given any $r \in \mathbb N$ and any fixed graph $H$, there is $n \in \mathbb N$ such that every $r$-edge-colouring of the complete graph $K_n$ contains a monochromatic copy of $H$. We write $R_r(H)$ to denote the smallest~$n$ for which the above holds, and set $R(H):=R_2(H)$.

Since the 1930s, there has been significant interest in determining the values of such \textit{Ramsey numbers} $R_r(H)$. 
The most famous subcase of this problem is the case when $r=2$ and $H$ is a complete graph. The value of $R(K_k)$ is only known exactly when $k \leq 4$ and thus there has been a focus on general upper and lower bounds for $R(K_k)$.
In particular, a famous probabilistic argument of Erd\H{o}s~\cite{erdos1947}  yields the bound $R(K_k) \geq 2^{k/2}$, whilst  Erd\H{o}s and Szekeres~\cite{es}
proved that $R(K_k) \leq 4^k$. After a number of improvements to this latter upper bound~\cite{conlon, sah, thom}, Campos, Griffiths, Morris and Sahasrabudhe~\cite{cgms} recently gave the first exponential improvement: there is an $\eps >0$ such that  $R(K_k) \leq (4-\eps)^k$; see also~\cite{balisteretal} for a recent general upper bound on $R_r(K_k)$. The method of~\cite{cgms} was then optimised by Gupta, Ndiaye, Norin and Wei~\cite{gnnw} who proved that $R(K_k) \leq 3.8^k$.

In general, there are relatively few graphs $H$ for which the exact value of $R_r(H)$ is known.
One class of graphs where we do have a better understanding of (two-colour) Ramsey numbers are so-called tilings.
For a fixed graph $H$, an {\it $H$-tiling} is a collection of \hbox{vertex-disjoint} copies of $H$.
For $m\in\mathbb N$, we write $mH$ to denote an $H$-tiling consisting of $m$ copies of $H$.
In the late 1960s,
Erd\H os~\cite[Problem 9]{erdosproblem} raised the question of determining $R(mK_{k})$ for $k \geq 3$.
The following well-known result of Burr, Erd\H{o}s and Spencer~\cite{BurrES} answers this question for sufficiently large $H$-tilings more generally.
\begin{theorem}[The Burr--Erd\H{o}s--Spencer theorem~\cite{BurrES}]\label{theorem:BES_symmetric}
For a fixed graph~$H$ without isolated vertices, there exist constants $C$ and $m_0$ such that $$R(mH)=\bigl(2|H|-\alpha(H)\bigr)m+C$$ provided $m\ge m_0$.
\end{theorem}
Here $|H|$ denotes the number of vertices in $H$ and $\alpha (H)$ denotes the size of the largest independent set in $H$.
In the case of triangles, Burr, Erd\H{o}s and Spencer~\cite{BurrES}  showed that in fact
$R(mK_3)=5m$ for every $m \geq 2$.

Given any graph~$H$ without isolated vertices, the following example (see Figure~\ref{fig:BES_symmetric}) shows that 
$R(mH)\geq \bigl(2|H|-\alpha(H)\bigr)m-1$.
Let~$G$ be a complete graph whose vertices are partitioned into two sets $X, Y$ with $|X|=|H|m-1$ and $|Y|=\bigl(|H|-\alpha(H)\bigr)m-1$.
All edges in~$X$ are coloured red, whereas all remaining edges are coloured blue.
Since~$H$ has no isolated vertices, a red copy of~$H$ must lie completely in~$X$.
On the other hand, a blue copy of~$H$ has at most $\alpha(H)$ vertices in~$X$ and thus at least~$|H|-\alpha(H)$ vertices in~$Y$.
Using these observations, it is easy  to check that~$G$ does not contain  a red copy of $mH$ or a blue copy of $mH$. 

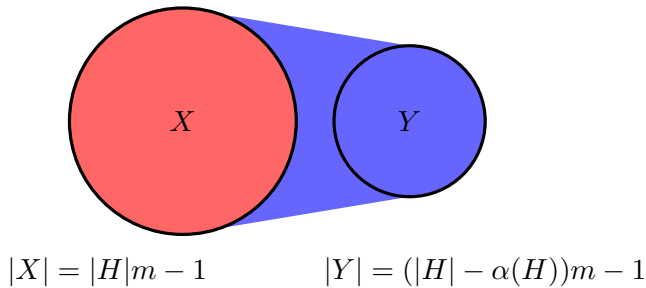
\begin{figure}[h!]
\centering
\begin{tikzpicture}
    \fill[fill=blue!60] (0,1.5)--(3,1)--(3,-1)--(0,-1.5)--(0,1.5);
    \filldraw[very thick, fill=red!60] (0,0) circle (1.5cm);
    \filldraw[very thick, fill=blue!60] (3,0) circle (1cm);
    \node at (0,0) {$X$};
    \node at (3,0) {$Y$};
    \node at (-1,-2) {$|X|=|H|m-1$};
    \node at (4,-2) {$|Y|=(|H|-\alpha(H))m-1$};
\end{tikzpicture}
\caption{The extremal construction of Theorem~\ref{theorem:BES_symmetric}.}
\label{fig:BES_symmetric}
\end{figure}

Burr~\cite{Burr}, and subsequently Buci\'c and Sudakov~\cite{BucicS}, provided methods for computing the value of~$C$ in Theorem~\ref{theorem:BES_symmetric} exactly.
Buci\'c and Sudakov~\cite{BucicS} also obtained the current best bounds for~$m_0$. In the case of $K_k$-tilings, their work states that there is a constant $D>0$ such that $R(mK_k)=(2k-1)m+R(K_{k-1})-2$ provided $m\ge 2^{Dk}$. Moreover, the bound on $m$ is essentially tight; see~\cite{BucicS}.
Sulser and Truji\'c~\cite{sulser} obtained better bounds on $m_0$ in Theorem~\ref{theorem:BES_symmetric} in the case when $H$ is sparse (e.g., a graph of bounded maximum degree).
There has also been interest in analogues of Theorem~\ref{theorem:BES_symmetric} in the setting of graphs of large minimum degree; see~\cite{balogh}.

\subsubsection{The multi-colour case}

An aim of this paper is to provide  analogues of Theorem~\ref{theorem:BES_symmetric} for more colours.
More precisely, for every graph $H$ and $m \in \mathbb N$, we 
determine the value of $R_3(mH)$ up to an additive constant. In fact, we  determine the value of 
$R_r(mH)$ up to an additive constant for many choices of $r$ and $H$.

It turns out the most interesting case of this problem is when $H$ has a bipartite component. 
In the case when $H$ has no bipartite component, the following result 
yields the value of $R_r (mH)$, up to a constant, \textit{for all}
$r\geq 3$.

\begin{theorem}\label{thm:rambi}
    Let $r \geq 3$ be an integer and let $H$ be a graph with no bipartite component.
    There exist constants $C=C(r,H)$ and $m_0=m_0(r,H)$ such that
    $$R_r(mH)=mr|H|+C$$
    provided $m\ge m_0$. 
    Furthermore, $-r+1\le C\le R_r(H)-r|H|$.
\end{theorem}
Note that existing arguments in the literature yield a slightly weaker version of Theorem~\ref{thm:rambi}  (see, e.g., \cite[Theorem~1.11]{GyarfasSS}). For completeness, we prove Theorem~\ref{thm:rambi}
in the appendix.

As the reader will see in the proof of Theorem~\ref{thm:rambi}, the bound on 
$R_r(mH)$ is not too difficult to obtain (the main subtlety is ensuring equality in the displayed equation in Theorem~\ref{thm:rambi}). In particular, note that in any large enough $r$-edge-coloured complete graph~$G$ one can easily (through repeated use of Ramsey's theorem) obtain a collection of vertex-disjoint monochromatic copies of~$H$
covering almost all of the vertices of $G$. As at least a $1/r$-proportion of these copies of~$H$ have the same colour, this simple argument implies that
$R_r(mH) \leq mr|H|+R_r(H)-1$.

\smallskip

As we will now explain, to understand the behaviour of $R_r (mH)$  when~$H$ has a bipartite component, things are more involved and we 
require the following definition.

\begin{define}\label{definition:extremal_construction}
    Given a graph $H$ and $m,r\in\mathbb N$, we let $f_r(m,H)$ denote the largest $n\in\mathbb N$ such that there exists an $n$-vertex $r$-edge-coloured complete graph $G$ on vertex set $V(G)=V_1\cup\dots\cup V_r$ satisfying the following:
\begin{itemize}
    \item $G$ does not contain a monochromatic copy of $mH$;
    \item $G[V_i]$ is a monochromatic clique for every $i\in[r]$;
    \item $G[V_i,V_j]$ is a monochromatic complete bipartite graph for every distinct $i,j\in[r]$.
\end{itemize}
\end{define}

\cref{definition:extremal_construction} is motivated by the extremal construction of \cref{theorem:BES_symmetric} (see Figure~\ref{fig:BES_symmetric}) and extremal constructions for Theorem~\ref{thm:rambi}.
In particular, it is easy to check that~$f_2(m,H)=\bigl(2|H|-\alpha(H)\bigr)m-2$ and so $R_2(mH)\le f_2(m,H)+C$ for some constant $C=C(H)$.
In general, we clearly have $R_r(mH)\geq f_r(m,H)+1$ for all choices of $H,r,m$.

Note though it may not be the case that $R_r(mH)=f_r(m,H)+1$ for a given choice of $H,r,m$.
For example, the aforementioned work of Burr, Erd\H{o}s and Spencer~\cite{BurrES}
yields $R_2(mK_3)=5m=f_2(m,K_3)+2$ for all $m \geq 2$. However,
our first main result states that~$R_3(mH)$ and~$f_3(m,H)$ must differ by at most a constant.

\begin{theorem}\label{theorem:r=3}
    Let~$H$ be a graph containing at least one edge.
    There exists a constant~$C=C(H)$ such that for every $m\in\mathbb N$ we have
    $$|R_3(mH)-f_3(m,H)|\le C.$$
\end{theorem}

Our next result asserts that~$R_r(mH)$ and~$f_r(m,H)$ differ by a constant for arbitrary $r\in\mathbb N$, provided the chromatic number~$\chi(H)$ of~$H$ is large enough.

\begin{theorem}\label{theorem:chi>=r}
    Let~$r\ge 2$ be an integer and let~$H$ be a graph with $\chi(H)\ge r$.
    There exists a constant~$C=C(H,r)$ such that for every $m\in\mathbb N$ we have
    $$|R_r(mH)-f_r(m,H)|\le C.$$
\end{theorem}

One should interpret Theorems~\ref{theorem:r=3} and~\ref{theorem:chi>=r} as follows: in these cases the Ramsey problem has a (near) extremal example which takes the form of a graph $G$ as in Definition~\ref{definition:extremal_construction}. In fact, we prove stronger versions of these theorems in Section~\ref{sec:5}. Indeed, roughly speaking Theorems~\ref{theorem:r=3'} and~\ref{theorem:chi>=r'} tell us that for these cases, the Ramsey problem has a (near) extremal example which not only takes the form of a graph $G$ as in Definition~\ref{definition:extremal_construction}, but further in $G$ the largest $H$-tiling in each of the colour classes has the \textit{same size}.
Notice this resonates with the extremal construction for Theorem~\ref{theorem:BES_symmetric} stated on page 2.

\medskip

We obtain a similar result for arbitrary $r$ when $H$ is a complete bipartite graph.
 Given any graph~$F$, let $\sigma(F)$ denote the size of the smallest possible colour class in any proper $\chi(F)$-colouring of the vertices of $F$.

\begin{define}\label{definition:extremal_construction_bi}
    Given a complete bipartite graph $H$ and $m,r\in\mathbb N$, we let $b_r(m,H)$ denote the largest $n\in\mathbb N$ such that there exists an $n$-vertex $r$-edge-coloured complete graph $G$ satisfying the following two conditions:
    \begin{itemize}
        \item[(i)]  There exist (not necessarily disjoint) vertex sets $X_1,\dots, X_r, Y_1,\dots,Y_r \subseteq V(G)$, where $V(G)= X_1 \cup \dots \cup X_r \cup Y_1 \cup \dots \cup Y_r$ and such that every edge in $G$ coloured  with $i\in[r]$ is either incident to $X_i$ or lies completely in $Y_i$.
        \item[(ii)]  For every $i\in[r]$ we have 
    $$|X_i|\cdot \frac{|H|}{\sigma(H)}+|Y_i|< m|H|.$$
    \end{itemize}
\end{define}

\begin{remark}
    It is easy to see that the value of $b_r(m,H)$ does not change if we further impose that the sets $X_1,\dots,X_r$ are pairwise disjoint and $X_i\cap Y_j=\emptyset$ for every $i,j\in[r]$.
    On the other hand, all edges not incident to $X_1\cup\dots\cup X_r$ must lie in $E(G[Y_1])\cup\dots\cup E(G[Y_r])$.
\end{remark}

\begin{remark}
    By the previous remark, once we fix the sizes of the sets $X_1,\dots, X_r$ and $Y_1,\dots, Y_r$, determining the maximum cardinality of $|V(G)|$ is equivalent to determining the largest clique whose edges can be covered using $r$ cliques of sizes $|Y_1|,\dots,|Y_r|$.
    A symmetric version of this question has been considered by Erd\H{o}s and R{\' e}nyi \cite{ErdosRenyi} and by Hor{\'a}k and Sauer \cite{clique-cover}.
    There is a vast literature on clique-covering problems, see e.g., the survey by Pullman \cite{Pullman}.
\end{remark}

Observe that a graph $G$ as in \cref{definition:extremal_construction_bi} does not contain a monochromatic copy of $mH$.
Indeed, a monochromatic copy of $H$ of colour $i$ with at least one vertex in $V(G)\setminus (X_i\cup Y_i)$ has at least $\sigma(H)$ vertices in $X_i$.
It follows that a monochromatic copy of $mH$ of colour $i$ has at most $(|H|-\sigma(H))|X_i|/\sigma(H)$ vertices in $V(G)\setminus (X_i\cup Y_i)$ and thus at most
$$|X_i|+|Y_i|+\frac{|H|-\sigma(H)}{\sigma(H)}\cdot|X_i|=|X_i|\cdot \frac{|H|}{\sigma(H)}+|Y_i|< m|H|$$
vertices in total, a contradiction.
In particular, we have $R_r(mH)\ge b_r(m,H)+1$.
Our next theorem states that this bound is tight up to a linear error term.

\begin{theorem}\label{thm:bi}
    Let~$r\ge3$ be an integer and let~$H$ be a complete bipartite graph.
    For every $\eps>0$ there exists $m_0\ge1$ such that
    $$R_r(mH)\le(1+\eps)\cdot b_r(m,H)$$
    whenever $m\ge m_0$.
\end{theorem}

{\noindent \textbf{Remark.}}
In this version of the paper, given a fixed graph $H$ with a bipartite component, we have not discussed what precise $r$-edge-colourings maximise $f_r(m,H)$. Similarly, we have not discussed which precise $r$-edge-colourings maximise $b_r(m,H)$. This will be discussed in an updated version of the paper.

\smallskip

{\noindent \textbf{Additional note.}}
In simultaneous and independent work, 
 Liu, Turevskii, Wang and  Yan~\cite{liuetal} have obtained similar results.
 More precisely, they proved that for any integer~$r\ge2$ and connected graph~$H$, the Ramsey number~$R_r(mH)$ and~$f_r(m,H)$ differ by~$o(m)$.
 Thus, their result applies to  all choices of~$r$ and connected~$H$, in exchange for a linear error term (rather than an additive constant \'a la Theorems~\ref{theorem:r=3} and~\ref{theorem:chi>=r}).
In particular, this means their work does not imply 
 Theorems~\ref{theorem:r=3} and~\ref{theorem:chi>=r}.
 
 Interestingly, a consequence of their result and our \cref{thm:bi} is that the functions~$f_r(m,H)$ and~$b_r(m,H)$ differ by~$o(m)$ when~$H$ is complete bipartite, which is not obvious from their respective definitions.
 
 We remark that the reason we can obtain results that are best-possible up to additive constant terms is our notion of {\it $(H,r)$-gadgets} which is introduced in Section~\ref{section:gadgets}.

\subsection{Generalisations of the random Ramsey theorem}
Another aim of the paper is to provide random analogues of some of the results mentioned in Section~\ref{sec11}.

The \emph{random graph $G_{n,p}$}
has vertex set $[n]:=\{1,\dots, n\}$ where each possible edge is present with probability $p$, independently of all other edges.
The celebrated \textit{random Rado theorem} of
R\"odl and Ruci\'nski~\cite{random1, random2, random3} essentially determines the values of $p$ for which $G_{n,p}$ is asymptotically almost surely (a.a.s.) $(H,r)$-Ramsey.
 To state their result formally we require a few definitions.
  Given a graph $H$,
we say that $G$ is \emph{$(H,r)$-Ramsey} if every $r$-edge-colouring of  $G$ yields a monochromatic copy of $H$ in~$G$. 
Given a graph $H$, set $d_2(H):=0$ if $e(H)=0$; $d_2(H):=1/2$ when $H$ is precisely an edge and define $d_2(H) := (e(H)-1)/(|H|-2)$ otherwise.
Then define $m_2(H) := \max_{H' \subseteq H}d_2(H')$ to be the \emph{$2$-density} of $H$.

\begin{theorem}[The random Ramsey theorem~\cite{random1, random2, random3}]\label{randomramsey} Let $r \geq 2$ 
 be a positive integer and let $H$ be a graph that is not a forest consisting of stars and paths of length $3$.
There are positive constants $c,C$ such that
$$\lim _{n \rightarrow \infty} \mathbb P \bigl[G_{n,p} \text{ is } (H,r)\text{-Ramsey}\bigr]=\begin{cases}
0 &\text{ if } p \leq cn^{-1/m_2(H)}; \\
 1 &\text{ if } p \geq Cn^{-1/m_2(H)}.\end{cases}$$
\end{theorem} 
We refer to the $p \leq cn^{-1/m_2(H)}$ case of Theorem~\ref{randomramsey}
as the \textit{$0$-statement} and the  $p \geq Cn^{-1/m_2(H)}$ case of Theorem~\ref{randomramsey} as the \textit{$1$-statement}.

There has also been interest in transferring other Ramsey-type results to the setting of random graphs. For example, a classical result of Gerencs\'er and Gy\'arf\'as~\cite{GG}
states that every $2$-edge-coloured $K_n$ contains a monochromatic copy of a path 
$P_{2n/3}$ on $2n/3$ vertices (and the bound on the length of the path is best-possible): resolving a question of Dudek and Pra\l at~\cite{dud}. In 2016, Letzter~\cite{letz}
showed that, provided $pn\rightarrow \infty$, a.a.s.\@ $G_{n,p}$ contains
a monochromatic copy of $P_{(2/3+o(1))n}$ whenever it is  $2$-edge-coloured.
Analogous results for monochromatic cycles of linear size  were obtained in~\cite{aps, kkm}.

The following result significantly strengthens 
the $1$-statement of 
Theorem~\ref{randomramsey} and provides random analogues of the Burr--Erd\H{o}s--Spencer theorem (Theorem~\ref{theorem:BES_symmetric}),  \cref{thm:rambi}, \cref{theorem:r=3} and \cref{theorem:chi>=r}.
For any integer $r\ge2$ and graph $H$, we let $c_r(H)$ denote the limit of $f_r(m,H)/(m|H|)$ as $m\to\infty$.\footnote{The existence of this limit in the cases we consider follows immediately from the statements of Theorems~\ref{theorem:r=3'} and~\ref{theorem:chi>=r'}. More generally, it is not difficult to show this limit  exists by Definition~\ref{definition:extremal_construction}.}

\begin{theorem}\label{randomramsey2} Let $r \geq 2$ 
 be a positive integer and let $H$ be a graph without isolated vertices and 
 $m_2(H) \geq 1$.
 Given any $\gamma >0$, there exists a  positive constant $C=C(H,r, \gamma)$  so that if
 $p \geq Cn^{-1/m_2(H)}$, then
  a.a.s.\@ every $r$-edge-colouring of $G_{n,p}$ contains the following:
 \begin{itemize}
     \item[$(\alpha_1)$] an $H$-tiling $\mathcal H_1$ covering at least $(1-\gamma)n$ vertices where each copy of $H$ in $\mathcal H_1$ is monochromatic;
     \item[$(\alpha_2)$] if $r=2$, a monochromatic $H$-tiling $\mathcal H_2$ covering at least 
     $\left(\frac{|H|}{2|H|-\alpha (H)}-\gamma \right )n$ vertices;
     \item[$(\alpha_3)$] if $r=3$, a monochromatic $H$-tiling $\mathcal H_3$ covering at least $\left(\frac{1}{c_3(H)}-\gamma \right )n$ vertices; 
     \item[$(\alpha_4)$] if $\chi(H)\ge r$, a monochromatic $H$-tiling $\mathcal H_4$ covering at least $\left(\frac{1}{c_r(H)}-\gamma \right )n$ vertices. 
 \end{itemize}
\end{theorem} 

{\bf Remark.}
The cases $(\alpha_1)$ and $(\alpha_2)$ of Theorem~\ref{randomramsey2} were proved in our earlier preprint~\cite[Theorem~1.3]{arxivnote} which we made available on arXiv.
Since posting that note on arXiv, we decided to combine our work into a single paper.

\medskip

To emphasise, in $(\alpha_1)$ the copies of $H$ in $\mathcal H_1$ are each monochromatic, though different copies of $H$ can have different colours. In $(\alpha_2)$, all copies of $H$ in $\mathcal H_2$ have the same colour, and similarly for $\mathcal H_3$ in $(\alpha_3)$ and $\mathcal H_4$ in $(\alpha_4)$.

Roughly speaking, the $0$-statement of the random Ramsey theorem tells us that a typical $n$-vertex graph of edge density significantly less than $n^{-1/m_2(H)}$ has an $r$-edge-colouring that does not contain a single monochromatic copy of $H$.
On the other hand, for $H$ as in Theorem~\ref{randomramsey2}, $(\alpha_1)$ tells us that a typical $n$-vertex graph of edge density significantly more than $n^{-1/m_2(H)}$ 
has the property that, however it is $r$-edge-coloured, 
not only does it have a mononochromatic copy of $H$ (\'a la Theorem~\ref{randomramsey}) but in fact
it 
can almost be completely covered by vertex-disjoint monochromatic copies of $H$.

Further, $(\alpha_2)$ yields a random analogue of the Burr--Erd\H{o}s--Spencer theorem (Theorem~\ref{theorem:BES_symmetric}). In fact, the $p=1$
case of $(\alpha_2)$ is simply a rephrasing of an asymptotic version of Theorem~\ref{theorem:BES_symmetric}. In particular, note that 
Theorem~\ref{theorem:BES_symmetric} implies that there is a constant $K$ such that, for every sufficiently large $n$, there is an
$r$-edge-colouring of $K_n$ for which the largest monochromatic $H$-tiling in $K_n$
covers at most $\left(\frac{|H|}{2|H|-\alpha (H)} \right )n+K$ vertices.
Thus, the bound on the size of $\mathcal H_2$ in $(\alpha_2)$ is asymptotically sharp.
In particular, for  $H$ as in the statement of Theorem~\ref{randomramsey2}, $(\alpha_2)$ essentially tells us that once~$p$ is sufficiently large enough to ensure a single monochromatic of $H$ in $G_{n,p}$, one actually obtains a monochromatic $H$-tiling of `largest attainable' size.
Similarly, $(\alpha_3)$ and $(\alpha_4)$ yield random analogues of \cref{theorem:r=3} and \cref{theorem:chi>=r} and the bounds on the size of $\mathcal H_3$ and $\mathcal H_4$ are asymptotically sharp.
Moreover, since at least a $1/r$-proportion of the copies of $H$ in $\mathcal H _1$ have the same colour, ($\alpha_1$) yields a random version of Theorem~\ref{thm:rambi}.

\smallskip

The phenomenon described above does not extend to those graphs $H$ with 
 $m_2(H)< 1$,
however. Indeed,
 if $H$ does not contain isolated vertices, then note  that $m_2(H)<1$ implies $H$ is a matching.
Consider the case when $H=K_2$ and thus $m_2(H)=1/2$. A construction of Cockayne and Lorimer~\cite{CockayneL} shows that there is an $r$-edge-colouring
of $K_n$ for which the largest monochromatic $K_2$-tiling (i.e., monochromatic matching) contains at most $n/(r+1)$ edges. Moreover, recently Gishboliner, Krivelevich and Michaeli~\cite[Theorem~4]{gkm}
showed that, given any $r \geq 2$ and $\gamma >0$, there exists a constant $C=C(r,\gamma)$ such that if $p \geq C/n$, then a.a.s.\@ $G_{n,p}$ contains a monochromatic matching
on $(1/(r+1)-\gamma)n$ edges whenever it is $r$-edge-coloured.
The bound on $p$ here is best-possible since given any $\eps >0$ there exists a constant $c=c(\eps)$ such that if $p \leq c/n$, then a.a.s.\@ $G_{n,p}$ does not contain 
a matching of size $\eps n$ (let alone a monochromatic one); see, e.g.,~\cite[Theorem~4.9]{janson}.
On the other hand, the threshold for $G_{n,p}$ to be $(K_2,r)$-Ramsey is
$p=1/n^2$ since
if $p n^2\rightarrow \infty$ then a.a.s.\@ $G_{n,p}$
contains an edge (and thus a monochromatic copy of $K_2$ however the edges are coloured). Similarly, 
if $H$ is a fixed matching then the 
threshold for $G_{n,p}$ to be $(H,r)$-Ramsey is
$p=1/n^2$.

In Sections~\ref{sec26} and~\ref{sec:robust} we prove `robust' asymptotic versions of some of the results mentioned in Section~\ref{sec11}. These results, together with
 applications of the sparse regularity lemma~\cite{koh} and the K{\L}R conjecture
for random graphs~\cite{randomklr} yield 
Theorem~\ref{randomramsey2}.

{\noindent \textbf{Additional note.}}
The $(\alpha_2)$ case of Theorem~\ref{randomramsey2} has been simultaneously and independently proven
by Arag\~ao, Cheng, Filipe, Miyazaki, Peng and Yan~\cite{aragao}. Their proof is very different to ours, making use of the hypergraph container method.
The work in~\cite{liuetal} also  yields similar results for connected graphs $H$.

\smallskip

{\noindent \bf Notation.}
Throughout, $\mathbb N$ denotes the set of positive integers (i.e., it does not contain $0$).

Unless stated otherwise, if we consider a $2$-edge-colouring of a graph, then we assume the colours used are red and blue. On the other hand, if we consider 
an $r$-edge-colouring for $r \geq 3$ or when the value of $r$ is not specified, we
assume the colours used are $[r]$. Given an $r$-edge-coloured graph $G$,  for every $c\in[r]$, we write~$G_c$ to denote the subgraph of~$G$ induced by the edges coloured with~$c$.

We write $V(G)$ and $E(G)$ for the vertex and edge sets of a graph $G$, respectively, and define $|G|:=|V(G)|$ and $e(G):=|E(G)|$.
A set of vertices $S\subseteq V(G)$ is {\it independent} if no edge lies in it.
A subgraph~$H$ of~$G$ is {\it spanning} if~$V(H)=V(G)$.
Given a set~$X\subseteq V(G)$, we write~$G[X]$ for the {\it induced subgraph of~$G$ on~$X$}, that is, the subgraph with vertex set~$X$ which contains all edges of~$G$ lying in~$X$. Set $G\setminus X:=G[V(G)\setminus X]$.
Given a vertex $x \in V(G)$, we write $d_G(x,X)$ for the number of edges incident to $x$ in $G$ whose other endpoint is in $X$.

Given  disjoint $U,V\subseteq V(G)$, we  write~$G[U,V]$ to denote the bipartite
subgraph of $G$ with vertex set $U \cup V$ that has edge set equal  to the set of edges in $G$ with one endpoint in $U$ and the other endpoint in $V$.
If $\mathcal B$ is a collection of subgraphs of $G$, then $V(\mathcal B)$ denotes
the set of vertices that lie in at least one element of $\mathcal B$.

Given graphs $H$ and $G$, we say an $H$-tiling in $G$ is \textit{perfect} if it covers all of $V(G)$.
Given a graph~$H$ and~$k\in\mathbb N$, we write~$H(k)$ to denote the blow-up of~$H$ where every vertex is replaced by a class of~$k$ vertices.

Given a graph $H$, let $\sigma(H)$ denote the size of the smallest possible colour class in any proper $\chi(H)$-colouring of the vertices of $H$.

An \textit{equipartition} of a set $V$ is a partition of $V$ into classes whose sizes differ by at most $1$.

Constants in  hierarchies  are chosen from right to left.
For example, if we claim that a result holds whenever $0< a\ll b\ll c\le 1$, then 
there are non-decreasing functions $f:(0,1]\to (0,1]$ and $g:(0,1]\to (0,1]$ such that the result holds
for all $0<a,b,c\le 1$  with $b\le f(c)$ and $a\le g(b)$.  
Note that $a \ll b$ implies that we may assume in the proof that, e.g., $a < b$ or $a < b^2$.

\smallskip

{\noindent \bf Organisation of the paper.}
The paper is organised as follows.  In the next section we collect together a number of useful tools, including versions of Szemer\'edi's regularity lemma~\cite{sze},
the Poincar\'e--Miranda theorem and a robust version of Theorem~\ref{theorem:BES_symmetric}. In Section~\ref{section:gadgets} we introduce our crucial notion of $(H,r)$-gadgets.
We prove Theorems~\ref{theorem:r=3} and~\ref{theorem:chi>=r} in Section~\ref{sec:5}.
In Section~\ref{sec:bip} we prove (a significant strengthening of) Theorem~\ref{thm:bi}.
Robust versions of Theorems~\ref{theorem:r=3} and~\ref{theorem:chi>=r} are proven in Section~\ref{sec:robust}; these results are applied in the proof of Theorem~\ref{randomramsey2}, which is given in Section~\ref{sec:random}.

\section{Useful tools}

\subsection{{Coloured version of the  regularity lemma}}
We will use a multicolour version of Szemer\'edi's regularity lemma~\cite{sze}.
We first introduce some notation.
Given two disjoint sets $U,V$ of vertices in a graph $G$, the \emph{density}
$d(U,V)$ 
is defined as
$$d(U,V):=\frac{e(G[U,V])}{|U||V|}.$$
Given~$\eps,d>0$, a graph~$G$ and two disjoint sets~$U, V\subseteq V(G)$, we say that $G[U, V]$ is \emph{$[\eps, d]$-regular} if~$d(U,V)\ge d$ and, for all sets~$X\subseteq U$ and~$Y\subseteq V$ with~$|X| \ge \eps|U|$ and~$|Y| \ge \eps|V|$, we have~$|d(U,V) - d(X,Y)| < \eps$.

We will apply the following \textit{degree form}   of Szemer\'edi's regularity lemma that can be easily deduced from the multicoloured version, e.g., given in~\cite{ks}.

\begin{lemma}\label{theorem:regularity}
For every~$\eps>0$ and $r,\ell_0\in \mathbb N$ there is an~$M = M(\eps,\ell_0,r)\in \mathbb N$ such that the following holds.
For every $d\in[0,1]$ and every $r$-edge-coloured graph~$G$ on~$n\ge M$ vertices, there exists a partition~$V_0,V_1,\dots,V_\ell$ of $V(G)$ and a spanning subgraph~$G'$ of~$G$ with the following properties:
\begin{itemize}
    \item~$\ell_0 \le \ell \le M$ and~$|V_0| \le \eps n$;
    \item~$|V_i| = |V_1|~$ for every~$i\in [\ell]$;
        \item~$d_{G'}(x) \ge d_{G}(x) - (r d+\eps)n$ for all~$x\in V(G)$;
    \item for all~$i \in [\ell]$ the graph~$G'[V_i]$ is empty;
    \item for all~$1 \le i < j \le \ell$ and $c\in[r]$, the graph $G_c'[V_i, V_j]$  either has density $0$ or is~$[\eps,d]$-regular.
\end{itemize}
\end{lemma}

The \emph{reduced graph~$R$ of~$G$ with parameters~$\eps$,~$\ell_0$, $r$ and~$d$} is the graph with vertex set~$\{V_i:i\in[\ell]\}$ and in which~$V_i V_j$ is an edge precisely when~$G_c'[V_i, V_j]$ is~$[\eps, d]$-regular for some $c\in[r]$.
Furthermore, we assign colour~$c$ to $V_iV_j$; if~$G_c'[V_i, V_j]$ is~$[\eps, d]$-regular for more than one choice of $c\in[r]$, we select one such colour arbitrarily.
The following well-known and simple consequence of the regularity lemma states that the reduced graph almost inherits the minimum degree of the original graph.

\begin{proposition}\label{reducedgraph}
Let $r,\ell_0 \in \mathbb N$, let~$0<\eps,d,c<1$ and let~$G$ be an $r$-edge-coloured $n$-vertex graph with~$\delta(G)\geq cn$. 
If~$R$ is the reduced graph of~$G$ obtained by applying Lemma~\ref{theorem:regularity} with parameters~$\eps$,~$\ell_0, r$ and~$d$, then~$\delta(R)\geq(c-2\eps-r d)|R|$.\qed
\end{proposition}

We will also use the following well-known result; it is an immediate consequence of both the so-called key lemma~\cite{ks} as well as the blow-up lemma~\cite{blowup}.

\begin{lemma}\label{corollary:embedding}
Given a graph~$H$, $q\in\mathbb N$ and~$d>0$, there exist~$\eps > 0$ and~$m_0\in\mathbb N$ such that the following holds. 
Given a graph~$R$, let~$G$ be the graph obtained by replacing every vertex $v\in V(R)$ with a set $V_v$ of~$m\ge m_0$ vertices and replacing each edge of~$R$ with an $[\eps,d]$-regular pair.
If~$R$ contains a perfect $H$-tiling, then~$G$ contains an $H(q)$-tiling covering all but
at most~$\sqrt\eps|G|$ vertices of~$G$.\qed
\end{lemma}

    

\subsection{{Coloured version of the sparse regularity lemma}}
In the proof of Theorem~\ref{randomramsey2} we  apply a variant of the sparse regularity lemma~\cite{koh}. To state this result we require a few definitions.

Let $p, \eps >0$. Given two disjoint sets $U,V$ of vertices in a graph $G$, the density
\textit{$d_{p} (U,V)$ of edges between $U$ and $V$ with respect to $p$} is
$$
d_p(U,V):= \frac{e\bigl(G[U,V]\bigr)}{p|U| |V|}.
$$
A bipartite graph with vertex classes $U,V$ is \textit{$(\eps, p)$-regular} if, for every $U' \subseteq U$ and $V' \subseteq V$ with $|U'|\geq \eps |U|$ and $|V'| \geq \eps |V|$ 
we have that 
$\bigl|d_p(U',V')-d_p (U,V)\bigr| \leq \eps$. In this case we call $(U,V)$ an \textit{$(\eps, p)$-regular pair}.  The following well-known fact follows immediately from the definition of an $(\eps, p)$-regular pair.

\begin{fact}\label{fact1}
  Let $0<\eps <\alpha$ and $\eps ':= \max \{ \eps/\alpha, 2\eps \}$.
Suppose that $(U,V)$ forms an $(\eps, p)$-regular pair with $d:=d_p(U,V)$.
Suppose $U'\subseteq U$ and $V' \subseteq V$ such that $|U'| \geq \alpha |U|$ and
$|V'| \geq \alpha |V|$. Then $(U',V')$ is an $(\eps ',p)$-regular pair with 
$d':=d_p(U',V')$ such that 
$ |d' - d| \leq \eps $.\qed
\end{fact}

Let $D \geq 1$ and $0<\eta, p \leq 1$. A graph $G$ is \textit{$(\eta, p,D)$-upper-uniform} if, given any disjoint sets $U,V \subseteq V(G)$ with $|U|,|V| \geq \eta |G|$, we have that $d_p (U,V) \leq D$.

We will apply the following coloured version of the sparse regularity lemma (see, e.g.,~\cite[Theorem 5.2]{letz}).
\begin{lemma}\label{sparserl}
    For every $\eps >0$, $t, r \in \mathbb N$ and $D>1$, there exist 
    $\eta >0$ and $T \in \mathbb N$ such that for any $0< p\leq 1$, if
    $G_1,\dots, G_r$ are $(\eta, p,D)$-upper-uniform graphs on vertex set $V$, there is an
    equipartition $V_1,\dots, V_s$ of $V$ into $s$ parts, where $t \leq s \leq T$, and for which all but at most an $\eps$-proportion of the pairs $(V_i,V_j)$ 
    ($i \neq j \in [s]$) induce an $(\eps, p)$-regular pair in each of the graphs $G_1,\dots, G_r$.
\end{lemma}

\subsection{The K{\L}R conjecture for random graphs}
Let $\eps, p >0$ and $n \in \mathbb N$. Let $H$ be a graph on vertex set $\{1,\dots, k\}$ and let $\textbf{m} =(m_{ij})_{ij  \in E(H)}$ be a sequence of non-negative integers. 
We define $\mathcal G(H,n,\textbf{m},p ,\eps)$ to be the collection  of all graphs $G$ obtained in the following way. The vertex set of $G$ is the disjoint union 
of sets $V_1,\dots, V_k$ of size $n$. For each edge $ij \in E(H)$, we add to $G$ an $(\eps, p)$-regular pair with $m_{ij}$ edges between the pair $(V_i,V_j)$.

A \textit{canonical copy of $H$} in such a $G \in \mathcal G(H,n,\textbf{m},p ,\eps)$
is a copy of $H$ with vertices $v_1,\dots, v_k$ where $v_i \in V_i$ for each $i \in V(H)$ and $v_iv_j \in E(G)$ for every $ij \in E(H)$.

In the proof of Theorem~\ref{randomramsey2} we will make use of the following
special case of the K{\L}R conjecture for random graphs~\cite[Proposition~4.2]{randomklr} (see also~\cite{bms, st}).
\begin{proposition}  
\label{randomklr}
    Given any graph $H$ and any $ d >0$, there exists $\eps >0$ with the following property. For every $\eta >0$, there is a $C>0$ such that if $p \geq CN^{-1/m_2(H)}$, then a.a.s.\@ the following holds in~$G_{N,p}$.
  For every $n \geq \eta N$ and $\textbf{m}$ with $m_{ij} \geq dpn^2$ for all $ij \in E(H)$, every subgraph $G$ of  $G_{N,p}$ in $\mathcal G(H,n,\textbf{m},p ,\eps)$ contains a canonical copy of $H$.
\end{proposition}

\subsection{The Poincar{\'e}--Miranda theorem}
The following analytical result is key to the proof of Theorem~\ref{theorem:chi>=r} in particular.
\begin{theorem}[Poincar{\'e}--Miranda]
    Consider $n$ continuous, real-valued functions of $n$ variables, $f_1,\dots,f_n:[a,b]^n\to\mathbb R$. 
    Assume that for each variable $x_i$, the function $f_i$ is non-positive when $x_{i}=a$ and non-negative when $x_{i}=b$. Then there is a point in $[a,b]^n$ in which all functions are simultaneously equal to $0$.
\end{theorem}

\subsection{A theorem of Gy\'arf\'as and S\'ark\'ozy}

  We will use the following result of Gy\'arf\'as and S\'ark\'ozy~\cite{GyarfasS} in the proof of Theorem~\ref{theorem:r=3}.
        Here $S_s$ denotes the star with $s$ edges.
    \begin{theorem}[Gy\'arf\'as and S\'ark\'ozy~\cite{GyarfasS}]\label{theorem:Gyarfas_Sarkozy}
        For every integers $m_1\ge m_2,s\ge 1$ we have
        $$R(m_1K_2,m_2K_2,S_s)=2m_1+m_2-1.$$
    \end{theorem}

\subsection{A robust version of Theorem~\ref{theorem:BES_symmetric}}\label{sec26}
The next result is a dense graph version of Theorem~\ref{theorem:BES_symmetric}.
It will be applied in the proof of Theorem~\ref{randomramsey2} as well as Theorem~\ref{theorem:r=3'}.

\begin{lemma}\label{robust2colours}
Let $H$ be a graph.
Given any $\eta >0$, there exist $\delta >0$ and $n _0 \in \mathbb N$ such that the following holds. Suppose that $G$ is a $2$-edge-coloured graph on $n \geq n_0$ vertices with $e(G) \geq (1-\delta) \frac{n^2}{2}$.
Then $G$ contains a monochromatic $H$-tiling covering at least 
     $\left(\frac{|H|}{2|H|-\alpha (H)}-\eta \right )n$
     vertices.
\end{lemma}
For the proof of Lemma~\ref{robust2colours}, we use a variant of the notion of a bowtie from~\cite{BurrES}.
\begin{define}

         Given a graph $H$,  set $h:=|H|$ and $\alpha :=\alpha (H)$.
   Let $K(h,\alpha)$ be the graph obtained from $K_{h}$ by removing the edges of a copy of $K_{\alpha}$.  An \textit{$H$-bowtie}  is the $2$-edge-coloured graph obtained from  the union of two copies 
     $K$ and $K'$ of $K(h,\alpha)$ whose vertex sets intersect precisely in their independent sets of size $\alpha$; further, the edges of $K$ are coloured red and the edges of $K'$ are coloured blue.
\end{define}
Note that an $H$-bowtie has $2|H|-\alpha (H)$ vertices and contains both a red copy of $H$ and a blue copy of $H$.

\begin{proof}[Proof of Lemma~\ref{robust2colours}]
    Given $\eta >0$, define $\delta, \delta _1, \eta _1>0$ and $n_0, N \in \mathbb N$ such that
    \begin{align}\label{hieralmost}
            0< \frac{1}{n_0} \ll \delta \ll \delta _1 \ll \frac{1}{N} \ll \eta _1 \ll \eta , \frac{1}{|H|}
    \end{align}
and where $N$ is divisible by $|H|$.

    Consider any $2$-edge-coloured graph $G$ on $n \geq n_0$ vertices as in the statement of the lemma. Let $\mathcal B$ be a maximum size collection of vertex-disjoint
    $H$-bowties in $G$. If $\mathcal B$ contains at least 
        $\left(\frac{1}{2|H|-\alpha (H)}-\frac{\eta}{|H|} \right )n$
        $H$-bowties, then $G$ contains both a red $H$-tiling and a blue $H$-tiling, each covering at least $\left(\frac{|H|}{2|H|-\alpha (H)}-{\eta} \right )n$
        vertices.

        We may therefore assume that $|\mathcal B|<\left(\frac{1}{2|H|-\alpha (H)}-\frac{\eta}{|H|} \right )n$ and thus $G_1:=G\setminus V(\mathcal B)$
        contains at least $\eta n$ vertices. Set $n_1:=|G_1|$.
        Note that 
        $$e(G_1) \geq \binom{n_1}{2}-\frac{\delta n^2}{2} \stackrel{(\ref{hieralmost})}{\geq} (1 -\delta _1) \frac{n_1 ^2}{2}.
        $$
        As $\delta _1 \ll 1/N$, we can apply Tur\'an's theorem to find a copy of 
        $K_{4^N}$ in $G_1$ and thus, by Ramsey's theorem, a monochromatic copy of $K_N$ in $G_1$.
        Greedily repeating this process, we obtain a $K_N$-tiling $\mathcal K$ in $G_1$ covering all but at most $\eta _1 n$ vertices of $G_1$ and such that each copy
        of $K_N$ in $\mathcal K$ is monochromatic.

Suppose that $\mathcal K$ contains at least $\eta _1 n/N$ red copies of $K_N$ and at least $\eta _1 n/N$ blue copies of $K_N$. In this case, 
there exists a pair $K^1, K^2$ of copies of $K_N$  in $\mathcal K$ such that $K^1$ is red, $K^2$ is blue, and there are all possible edges between $K^1$ and $K^2$ in $G$.
Indeed, if not then there are at least
$$ \frac{\eta _1 n}{N} \times \frac{\eta _1 n}{N}  \stackrel{(\ref{hieralmost})}{> } \frac{\delta n^2}{2}$$
non-edges in $G$, a contradiction.

As $1/N \ll 1/|H|$, by, e.g., the Erd\H{o}s--Stone--Simonovits theorem, there exists a set $V_1$ of $|H|$ vertices in $V(K^1)$ and a set $V_2$ of $|H|$ vertices in $V(K^2)$
such that all edges between $V_1$ and $V_2$ are coloured the same in $G$.
It is now straightforward to see that $G[V_1\cup V_2]$ contains an $H$-bowtie $\hat{H}$. Adding $\hat{H}$ to $\mathcal B$ creates a larger collection of vertex-disjoint $H$-bowties in $G$, a contradiction to the maximality of $\mathcal B$.

Thus, we may assume that
 all but at most $\eta _1 n/N$ of the copies of $K_N$ in $\mathcal K$ have the same colour; without loss of generality, red. Each red copy of $K_N$ in $\mathcal K$  contains a spanning red  $H$-tiling and each $H$-bowtie in $\mathcal B$ contains a red copy of $H$.
Thus, $G$ contains a red $H$-tiling covering at least
\begin{align*}
     |H|\cdot |\mathcal B| +(|G_1| -2 \eta _1 n) \geq \left(\frac{|H|}{2|H|-\alpha (H)} \right )n -2\eta _1 n \stackrel{(\ref{hieralmost})}{\geq}
      \left(\frac{|H|}{2|H|-\alpha (H)}-\eta  \right )n
\end{align*}
vertices of $G$, as desired.
\end{proof}

\section{Gadgets}\label{section:gadgets}
\subsection{Fractional tilings, gadgets and a Ramsey lower bound}
In this section, we introduce some key notions that will be crucial in the proofs of \cref{theorem:r=3} and \cref{theorem:chi>=r}.

\begin{define}[Graph homomorphism]\label{definition:graph_homomorphism}
    Let~$G$ be a graph (possibly with loops) and~$H$ be a graph.
    A {\it graph homomorphism} from~$H$ to~$G$ is a function $\phi:V(H)\to V(G)$ such that $xy\in E(H)$ implies $\phi(x)\phi(y)\in E(G)$. 
    We write~$\text{Hom}(G:H)$ for the set of all graph homomorphisms from~$H$ to~$G$. 
\end{define}

Note that, in the above language, a copy of~$H$ in~$G$ corresponds to an injective graph homomorphism $\phi\in\text{Hom}(G:H)$.
Next we define a well-known relaxation of the notion of a tiling.

\begin{define}[Fractional $H$-tiling]\label{definition:fractional_tiling}
    Let~$G$ be a graph (possibly with loops) and~$H$ be a graph.
    Let~${\bf c}=\{c_v:v\in V(G)\}$ be a~$|G|$-tuple of non-negative real entries.
    We refer to ${\bf c}$ as the {\it capacity} and write $||{\bf c}||:=\sum_{v\in V(G)}c_v$. 
    A {\it fractional $H$-tiling} in $(G,{\bf c})$ is a weighting function
    $$w:\text{Hom}(G:H)\to\mathbb R^{\ge0}$$
    such that, for every vertex $v\in V(G)$, we have
    $$w(v):=\sum_{\phi\in \text{Hom}(G:H)}w(\phi)\cdot\bigl|\phi^{-1}(v)\bigr|\le c_v.$$
    The {\it size} of a fractional~$H$-tiling~$w$ is defined as $$w(G):=\sum_{v\in V(G)}w(v)=\sum_{\phi\in \text{Hom}(G:H)}w(\phi)\cdot|H|.$$
    We write~$\nu_H(G,{\bf c})$ for the largest size among all fractional $H$-tilings in~($G$,{\bf c}).\footnote{Note~$\nu_H(G,{\bf c})$ is well-defined as there exist fractional $H$-tilings achieving maximum size, see \cref{observation:linear_programming}.}
    If~$c_v=1$ for every $v\in V(G)$, we write~$\nu_H(G)$ for~$\nu_H(G,{\bf c})$.
\end{define}

We are now ready to introduce the key notion of $(H,r)$-gadgets.

\begin{define}[$(H,r)$-gadget]\label{definition:gadgets}
    Let $r\in\mathbb N$ and~$H$ be a graph.
    Let $F$ be an $r$-edge-coloured $r$-vertex complete graph equipped with a loop at each vertex.
    We say $F$ is an {\it $(H,r)$-gadget} if there exists an~$|F|$-tuple ${\bf c}=\{c_v:v\in V(F)\}$ such that~$\nu_H(F_i,{\bf c})=|H|$ for every $i\in[r]$.  (Recall that $F_i$ is the subgraph of $F$ spanned by the edges of colour $i$.)
    Sometimes, we refer to $(F,{\bf c})$ also as an $(H,r)$-gadget.\footnote{Note that there might be more than one choice of~${\bf c}$ for which $(F,{\bf c})$ is an $(H,r)$-gadget.}
\end{define}

The choice of~$F$ being a looped complete graph is designed to represent large simple graphs whose structure resembles that of \cref{definition:extremal_construction}.
The next definition helps establish this connection.

\begin{define}[Blow-up]\label{definition:blow_up}
    Let $r\in\mathbb N$ and $F$ be an $r$-edge-coloured graph (possibly with loops).
    A {\it (simple) blow-up}~$B$ of~$F$ is a graph (i.e., without loops) with vertex set $V(B)=\bigcup_{v\in V(F)} V_{v}$ so that
    \begin{itemize}
        \item if $uv\in E(F)$ and $u\neq v$ then $B[V_u,V_v]$ is a monochromatic complete bipartite graph of the same colour as $uv$; 
        \item if $F$ has a loop at $v$ then $B[V_v]$ is a monochromatic clique of the same colour as the loop.
    \end{itemize}
\end{define}

To help the reader absorb the above definitions, we describe an explicit $(H,r)$-gadget for any choice of~$H$ and~$r$.

\begin{example}[Lexicographic gadget]\label{example:lexicographic}
    Let $r\in\mathbb N$ and~$H$ be a graph.
    Let $F$ be the $r$-edge-coloured complete graph on vertex set $V=[r]$, equipped with a loop at each vertex, where the edge $ij\in E(F)$ is coloured with~$j$ for every $i\le j$.
    Then~$(F,{\bf c})$ is an $(H,r)$-gadget where we take $c_1:=|H|$ and $c_i:=|H|-\alpha(H)$ for every $i\ge2$.
    
    Indeed, observe that the only edge coloured with~$1$ is the loop on vertex~$1$.
    Hence there is a unique homomorphism $\phi_0\in\text{Hom}(F_1:H)$, namely that which sends all vertices of~$H$ to~$1$.
    Clearly $\nu_H(F_1,{\bf c})\le c_1=|H|$ and equality is attained by assigning weight~$1$ to~$\phi_0$.
    
    If $i>1$, note that every $\phi\in\text{Hom}(F_i:H)$ sends at most~$\alpha(H)$ vertices to $V(F)\setminus \{i\}$ (as no edge coloured~$i$ lies in this set) and so at least $|H|-\alpha(H)$ vertices to~$i$.
    It follows that $\nu_H(F_i,{\bf c})\le\frac{|H|}{|H|-\alpha(H)}\cdot c_i=|H|$. 
    Equality is attained by, e.g., assigning weight~$1$ to the homomorphism sending an independent set of size~$\alpha(H)$ to $1$ and all remaining vertices to~$i$ (and weight~$0$ to other homomorphisms).
\end{example}

Let~$F$ be as in \cref{example:lexicographic} and let~$B$ be the blow-up of~$F$ obtained by replacing the vertex~$i$ with a class of~$c_i\cdot m$ vertices for some~$m$.
It is easy to check that~$B$ contains a monochromatic copy of~$mH$ in each colour, but not a monochromatic copy of $(m+1)H$.

\smallskip

The next two propositions summarise the key features of $(H,r)$-gadgets.
Roughly speaking, consider the blow-up~$B$ of an $(H,r)$-gadget $(F,{\bf c})$ obtained by replacing each vertex~$v$ of $F$ with~$c_v\cdot m$ vertices.
On the one hand, slightly adjusting the parts of~$B$ will yield a graph not containing a monochromatic~$mH$ and thus give a lower bound for~$R_r(mH)$.
On the other hand, one can argue that a monochromatic copy of~$mH$ in~$B$ occurs {\it in every colour}.
This last property will be useful to set up the inductive arguments in the proofs of \cref{theorem:r=3} and \cref{theorem:chi>=r}.

\begin{proposition}\label{proposition:gadget_extremal}
    Let $r,m\in\mathbb N$ and~$H$ be a graph.
    Let $(F,{\bf c})$ be an $(H,r)$-gadget where ${\bf c}=\{c_v:v\in V(F)\}$.
    If $B$ is a simple blow-up of $F$ where each vertex $v\in V(F)$ is replaced by a class~$V_v$ of $\left\lfloor c_v\cdot (m-1)\right\rfloor$ vertices, then $B$ does not contain a monochromatic $mH$.
    In particular,
    $$R_r(mH)\ge |B|\ge||{\bf c}||\cdot m-||{\bf c}||-r.$$
\end{proposition}

\begin{proof}
    Suppose for a contradiction that $B_i$ contains a copy of $mH$ for some~$i\in[r]$  (recall that~$B_i$ denotes the subgraph of~$B$ spanned by the edges coloured~$i$).
    
    We think of this copy of~$mH$ in~$B_i$ as a collection $\mathcal C$ of~$m$ injective homomorphisms in $\text{Hom}(B_i:H)$.
    Given $\phi\in\text{Hom}(F_i:H)$, we say a homomorphism $\phi'\in\text{Hom}(B_i:H)$ {\it corresponds to $\phi$} if $\phi(x)=v$ implies $\phi'(x)\in V_v$ for every $x\in V(H)$.
    Consider the weighting $w:\text{Hom}(F_i:H)\to\mathbb R^{\ge 0}$ where $$w(\phi)=\frac{\#\,\{\phi'\in\mathcal{C}:\text{$\phi'$ corresponds to $\phi$}\}}{m-1}.$$
    In particular, if~$\phi'$ corresponds to~$\phi$ then the image $\text{Im}(\phi')$ has precisely $|\phi^{-1}(v)|$ elements in~$V_v$ for each $v\in V(F_i)$.
    For each $v\in V(F_i)$, we have
    \begin{align*}
        w(v)&=\sum_{\phi\in \text{Hom}(F_i:H)}w(\phi)\cdot\bigl|\phi^{-1}(v)\bigr|
        =\sum_{\phi'\in\mathcal C}\frac{\bigl|\text{Im}(\phi')\cap V_v\bigr|}{m-1}\le\frac{\bigl|V_v\bigr|}{m-1}\le c_i,
    \end{align*}
    where the penultimate inequality follows from the fact that~$\mathcal C$ is a collection of vertex-disjoint copies of~$H$ and so it covers at most~$|V_v|$ vertices in~$V_v$.
    Therefore, $w$ is a fractional $H$-tiling of~$F_i$.
    The size of~$w$ is
    \begin{align*}
        w(F_i)=\sum_{v\in V(F_i)} w(v)=\sum_{\phi\in\text{Hom}(F_i:H)}|H|\cdot w(\phi)=\frac{|H|\cdot |\mathcal C|}{m-1}=\frac{m|H|}{m-1}>|H|.
    \end{align*}
    This is a contradiction as $\nu_H(F_i,{\bf c})=|H|$.
\end{proof}

\begin{proposition}\label{proposition:gadget_induction}
    Let $r\in\mathbb N$ and~$H$ be a graph.
    Let $(F,{\bf c})$ be an $(H,r)$-gadget such that all entries of ${\bf c}=\{c_v:v\in V(F)\}$ are rational.
    Then there exists $m\in\mathbb N$ such that the following holds. 
    If $B$ is a simple blow-up of $F$ where each vertex $v\in V(F)$ is replaced by a class $V_v$ of $ c_v \cdot m$ vertices, then $B$ contains a monochromatic copy of $mH$ in each colour $i\in[r]$.
\end{proposition}

\begin{proof}
    For every $i\in[r]$, since ${\bf c}$ has rational entries, it follows from linear programming that there is a fractional $H$-tiling $w_i$ in $(F_i,{\bf c})$ such that $w_i(F_i)=\nu_H(F_i,{\bf c})=|H|$ and $w_i(\phi)$ is rational for each $\phi\in\text{Hom}(F:H)$.
    Pick $m\in\mathbb N$ such that $w_i(\phi)\cdot m$ is an integer for every $i\in[r]$ and $\phi\in\text{Hom}(F:H)$. Let $B$ be as defined in the statement of the proposition.
    It suffices to show that, for each $i\in[r]$,  $B_i$ contains a copy of $mH$.
    
    Fix $i\in[r]$. For each $\phi\in\text{Hom}(F:H)$, select $w_i(\phi)\cdot m$ distinct injective graph homomorphisms~$\phi'$ in $\text{Hom}(B_i:H)$ such that $\phi(x)=v$ implies $\phi'(x)\in V_v$ for every $x\in V(H)$.
    Note that the image $\text{Im}(\phi')$ has precisely $\bigl|\phi^{-1}(v)\bigr|$ elements in $V_v$ for each $v\in V(F_i)$.
    For every $v\in V(F_i)$, since $w_i$ is a fractional $H$-tiling $w_i$ in $(F_i,{\bf c})$, we have
    $$\sum_{\phi\in \text{Hom}(F_i:H)} \bigl|\phi^{-1}(v)\bigr|\cdot \bigl(w_i(\phi)\cdot m\bigr)=w_i(v)\cdot m\le c_v\cdot m=|V_v|.$$
    Therefore, we can select the injective graph homomorphisms $\phi'$ so that their images are pairwise disjoint.
    In other words, we have found an $H$-tiling in $B_i$ consisting of
    $$\sum_{\phi\in\text{Hom}(F_i:H)}w_i(\phi)\cdot m=\frac{m}{|H|}\sum_{v\in V(F_i)}w_i(v)=\frac{w_i(F_i)\cdot m}{|H|}=m$$
    copies of $H$, as required.
\end{proof}


\subsection{Existence of $(H,r)$-gadgets with rational capacity}\label{subsection:blow-ups}

\cref{proposition:gadget_induction} requires that the entries of~${\bf c}$ are rational.
This technical subsection is devoted to prove that, indeed, such~${\bf c}$ exists.

Intuitively, given an $(H,r)$-gadget $(F,{\bf c})$ where ${\bf c}$ has at least one irrational entry, we can slightly adjust the entries of~${\bf c}$ to obtain a new capacity~${\bf q}$ with rational entries.
Since the entries of~${\bf c}$ and~${\bf q}$ differ by little, the difference~$\bigl|\nu_H(F_i,{\bf c})-\nu_H(F_i,{\bf q})\bigr|$ should be small for every $i\in[r]$.
Since~$\nu_H(F_i,{\bf c})=|H|$, if we can argue that~$\nu_H(F_i,{\bf q})$ is a fraction whose denominator is greater than~$|\nu_H(F_i,{\bf c})-\nu_H(F_i,{\bf q})|$, this would imply~$\nu_H(F_i,{\bf q})=|H|$ for every $i\in[r]$ and thus that~$(F,{\bf q})$ is also an $(H,r)$-gadget.  

The next fact will be used to pick~${\bf q}$ appropriately.

\begin{fact}\label{fact:technical}
    Let $k,t\in\mathbb N$ and $c_1,\dots,c_t\ge0$.
    There exist integers $n_1,\dots, n_t\ge 0$ and $N\ge1$ such that, for every $i\in[t]$,
    $$\left|c_i-\frac{n_i}{N}\right|\le\frac{1}{kN}.$$
\end{fact}
    
    \begin{proof}
    For every $i\in[t]$, pick a decimal representation of~$c_i$.  
    By the pigeonhole principle, there exist two disjoint intervals $A,B\subseteq \mathbb N$ of size $|A|=|B|=k$, say $A=[a,\dots,a+k-1]$ and $B=[b,\dots,b+k-1]$ with $a+k\le b$, such that, for every $i\in[t]$ and $j\in[k]$, the $(a+j-1)$th and $(b+j-1)$th fractional digit of $c_i$ are equal.

    It follows that the last~$k$ integer digits before the fractional part of $c_i\cdot 10^{b+k-1}$ and $c_i\cdot 10^{a+k-1}$ are identical, for every $i\in[t]$.
    Thus, there exists an integer~$n_i\ge0$ such that
    $$\Bigl|\bigl(10^{b+k-1}-10^{a+k-1}\bigr)c_i-n_i\cdot 10^k\Bigr|\le 1.$$
    Dividing both sides by $(10^{b+k-1}-10^{a+k-1})$ yields
    $$\left|c_i-\frac{n_i}{10^{b-1}-10^{a-1}}\right|\leq \frac{1}{10^{b+k-1}-10^{a+k-1}}\leq \frac{1}{k(10^{b-1}+10^{a-1})}.$$
    Hence, we recover the statement of the claim by taking $N:=10^{b-1}-10^{a-1}$.
\end{proof}

The next proposition is useful to bound the change in~$\nu_H(G,{\bf c})$ under a small variation of~${\bf c}$.

\begin{proposition}\label{proposition:capacity_diff}
    Let~$G$ be a graph (possibly with loops) and~$H$ be a graph.
    Let ${\bf c}=\{c_v:v\in V(G)\}$ and ${\bf d}=\{d_v:v\in V(G)\}$ be two $|G|$-tuples of positive entries such that $c_v=d_v$ for all but one vertex $u\in V(G)$.
    Then
    $$|\nu_H(G,{\bf c})-\nu_H(G,{\bf d})|\le\min\left\{\frac{||{\bf c}||}{c_u},\frac{||{\bf d}||}{d_u}\right\}\cdot|c_u-d_u|.$$
\end{proposition}

\begin{proof}
    Without loss of generality, assume $c_u>d_u$.
    Let $w_{c}$ be a fractional $H$-tiling in $(G,{\bf c})$ with $w_{c}(G)=\nu_H(G,{\bf c})$.
    We have $w_{c}(G)\ge \nu_H(G,{\bf d})$ since $c_u>d_u$.
    Next, we bound $w_{c}(G)$ from above.
    Let $w_0:\text{Hom}(G:H)\to\mathbb R^{\ge0}$ with
    $$w_0(\phi):=w_{c}(\phi)\cdot(d_u/c_u).$$
    Then we have $w_0(v)= w_c(v)\cdot(d_u/c_u)\le d_v$ for every $v\in V(G)$.
    Hence, $w_0$ is a fractional $H$-tiling in $(G,{\bf d})$ and so $w_0(G)\le \nu_H(G,{\bf d})$.
    On the other hand, we have 
    \begin{align*}
        0\le w_{c}(G)-w_{0}(G)&=\sum_{\phi\in\text{Hom}(G:H)} |H|\cdot(w_{c}(\phi)-w_0(\phi))=\sum_{\phi\in\text{Hom}(G:H)}|H|\cdot w_{c}(\phi)\cdot\frac{c_u-d_u}{c_u}\\
        &= w_c(G)\cdot\frac{c_u-d_u}{c_u}\le\frac{||{\bf c}||}{c_u}\cdot(c_u-d_u).
    \end{align*}
    We conclude that
    $$\nu_H(G,{\bf d})\le w_c(G)\le \nu_H(G,{\bf d})+\frac{||{\bf c}||}{c_u}\cdot(c_u-d_u)$$
    and so, since $w_c(G)=\nu_H(G,{\bf c})$,
    $$|\nu_H(G,{\bf c})-\nu_H(G,{\bf d})|\le\frac{||{\bf c}||}{c_u}\cdot(c_u-d_u).$$
    Note that
    $$\frac{||{\bf c}||}{c_u}=1+\frac{1}{c_u}\cdot\sum_{v\neq u} c_v=1+\frac{1}{c_u}\cdot\sum_{v\neq u} d_v\le1+\frac{1}{d_u}\cdot\sum_{v\neq u} d_v=\frac{||{\bf d}||}{d_u}.$$
    This concludes the proof.
\end{proof}

Finally, we use linear programming.

\begin{observation}[A linear programming perspective]\label{observation:linear_programming}
    In what follows, we treat~$G$, $H$ and~${\bf c}$ as fixed parameters.
    A weighting function $w:\hom(G:H)\to\mathbb R^{\ge0}$ can be viewed as an element of the set $(\mathbb R^{\ge 0})^{|\hom(G:H)|}$.
    From this perspective, the set $S\subseteq(\mathbb R^{\ge 0})^{|\hom(G:H)|}$ of fractional $H$-tilings is the set of feasible solutions to the linear constraints \{$w(v)\le c_v:v\in V(G)$\}, hence~$S$ is a convex polytope in $(\mathbb R^{\ge 0})^{|\hom(G:H)|}$.
    Since $S$ is non-empty (${\bf 0}\in S$) and bounded (the constraints imply $S\subseteq[0,\max\{c_v:v\in V(G)\}]^{|\hom(G:H)|}$), there exists an optimal solution to the maximisation problem $\max w(G)$.
    In particular, we have $\max w(G)=\nu_H(G,{\bf c})$.
\end{observation}

We are now ready to prove the main lemma of this section.

\begin{lemma}\label{lemma:gadget_induction}
    Let $r\in\mathbb N$ and~$H$ be a graph.
    Let $F$ be an $(H,r)$-gadget.
    Then there exists a $|G|$-tuple ${\bf q}=\{q_v:v\in V(G)\}$ with rational entries such that $(F,{\bf q})$ is an $(H,r)$-gadget. 
\end{lemma}

\begin{proof}
    Let $(F,{\bf c})$ be an $(H,r)$-gadget for some~${\bf c}=\{c_v:v\in V(F)\}$.
    If any entry of ${\bf c}$ is zero, say $c_u=0$, then $F\setminus \{u\}$ is also an $(H,r)$-gadget, and we may run the argument below  on it instead.
    Hence, we may assume all entries of ${\bf c}$ are positive.
    
    Pick a sufficiently large integer $k\ge1$.
    Apply \cref{fact:technical} to the $|F|$-tuple ${\bf c}$ to obtain non-negative integers $\{n_v:v\in V(F)\}$ and $N\ge1$.
    Consider the $|F|$-tuple ${\bf q}:=\{q_v=n_v/N:v\in V(F)\}$; so all entries of ${\bf q}$ are rational.
    It suffices to show that $\nu_H(F_i,{\bf q})=|H|$ for every $i\in[r]$.

    Since all entries of~${\bf c}$ are positive and~$k$ is sufficiently large, we may assume all entries of~${\bf q}$ are also positive.
    By repeatedly applying \cref{proposition:capacity_diff}, we have
    $$|\nu_H(F_i,{\bf c})-\nu_H(F_i,{\bf q})|\le\sum_{v\in V(F_i)}\frac{||{\bf c}||}{c_v}\cdot|c_v-q_v|.$$
    Thus, by \cref{fact:technical} we have
    \begin{equation}\label{eq:diff_c_q}
        |\nu_H(F_i,{\bf c})-\nu_H(F_i,{\bf q})|\le\frac{1}{kN}\sum_{v\in V(F_i)}\frac{||{\bf c}||}{c_v}.
    \end{equation}

    Now, consider the linear program described in \cref{observation:linear_programming} with respect to~$F_i$, $H$ and~${\bf q}$.
    There exists a so-called {\it optimal basic feasible solution}, that is, a fractional $H$-tiling of size $\nu_H(F_i,{\bf q})$ whose non-zero weights are the unique solution to an equation $A{\bf x}={\bf b}$ representing a system of equations $\{w(v)= q_v:v\in S\}$ for some subsets $S\subseteq V(F_i)$.
    In particular, ${\bf x}=A^{-1}{\bf b}$.
    Since~${\bf b}$ is a subvector of~${\bf q}$, its entries can be expressed as fractions with denominator~$N$.
    As $A$ is a matrix with integer entries between~$0$ and~$|H|$, the entries of $A^{-1}$ are fractions with denominator $\det(A)$ by Cramer's rule.
    Therefore, we conclude that the entries of ${\bf x}$, i.e., the non-zero weights of the optimal basic feasible solution can be written as fractions with denominator~$N\cdot\det(A)$. 
    Thus, $\nu_H(F_i,{\bf q})$ can be written as a fraction with denominator $N\cdot\det(A)$.
    
    Since~$k$ was picked sufficiently large, we have\footnote{Note that 
    although ${\bf q}$ and thus {\bf b} depend on $k$, $\det(A)$ does not in the following sense. At the start of the proof one already knows that $A$ must be a matrix with entries from $\{0, \dots, |H|\}$, with $\det (A) \neq 0$ and whose dimension is bounded. Thus, we can choose $k$ sufficiently large so that
    $1/k \ll 1/\det(A)$ for all such candidate matrices $A$.
Hence, we indeed can ensure  (\ref{newfootnote}) holds.} 
    \begin{align}\label{newfootnote}
    \frac{1}{N\cdot\det(A)}>\frac{1}{kN}\sum_{v\in V(F_i)}\frac{||{\bf c}||}{c_v}.
    \end{align}
    The above inequality together with~\eqref{eq:diff_c_q} and $\nu_H(F_i,{\bf c})=|H|$ implies $\nu_H(F_i,{\bf q})=|H|$.
    This holds for every $i\in[r]$, and so we are done.
\end{proof}

\section{Proofs of \cref{theorem:r=3} and \cref{theorem:chi>=r}}\label{sec:5}

We prove the following generalisations of \cref{theorem:r=3} and \cref{theorem:chi>=r}.

\begin{theorem}\label{theorem:r=3'}
    Given any graph $H$ containing at least one edge, there exists a constant~$C=C(H)$ such that for every $m\in\mathbb N$
    $$|R_3(mH)-\max\{||{\bf c}||:(F,{\bf c}) \text{ is an $(H,3)$-gadget}\}\cdot m|\le C.$$
\end{theorem}

\begin{theorem}\label{theorem:chi>=r'}
    Given $r\geq 2$ and a graph $H$ with $\chi(H)\ge r$, there exists a constant~$C=C(H,r)$ such that for every $m\in\mathbb N$
    $$|R_r(mH)-\max\{||{\bf c}||:(F,{\bf c}) \text{ is an $(H,r)$-gadget}\}\cdot m|\le C.$$
\end{theorem}

To see why \cref{theorem:r=3'} and \cref{theorem:chi>=r'} imply \cref{theorem:r=3} and \cref{theorem:chi>=r} respectively, consider the following.
By \cref{proposition:gadget_extremal} and the definition of $(H,r)$-gadgets, for every graph~$H$ and $m,r\in\mathbb N$ we have 
\begin{equation}\label{eq:lower_bound}
R_r(mH)>f_r(m,H)\ge\max\{||{\bf c}||:(F,{\bf c}) \text{ is an $(H,r)$-gadget}\}\cdot m-C_0,
\end{equation}
where $C_0:=\max\{||{\bf c}||:(F,{\bf c}) \text{ is an $(H,r)$-gadget}\}+r$ depends only on~$r$ and~$H$.
Inequality~\eqref{eq:lower_bound} and \cref{theorem:r=3'} imply that~$R_3(mH)$ and $f_3(mH)$ differ by a constant independent of~$m$, which is precisely the statement of \cref{theorem:r=3}.
Similarly,~\eqref{eq:lower_bound} and \cref{theorem:chi>=r'} imply \cref{theorem:chi>=r}.

Furthermore, observe that~\eqref{eq:lower_bound} provides an appropriate lower bound on~$R_r(mH)$ for  \cref{theorem:r=3'} and \cref{theorem:chi>=r'}.
Thus, to prove these theorems, it suffices to show a corresponding upper bound.
We start with \cref{theorem:chi>=r'}.

\subsection{Proof of \cref{theorem:chi>=r'}} 

As we have just discussed, it suffices to show that
\begin{align}\label{target1}
R_r(mH)\le \max\{||{\bf c}||:(F,{\bf c}) \text{ is an $(H,r)$-gadget}\}\cdot m + C
\end{align}
for some constant $C=C(H,r)$.

We briefly describe our proof strategy.
We proceed by induction on~$m$.
Given an $r$-edge-coloured complete graph with more vertices than the upper bound in (\ref{target1}), we try to find a blow-up~$B$ of an $(H,r)$-gadget $(F,{\bf c})$.
If we succeed, we remove~$B$ from~$G$ and find a large monochromatic $H$-tiling in the remaining graph by the inductive hypothesis.
We then use \cref{proposition:gadget_induction} to find an $H$-tiling of the same colour in~$B$.
The union of these two $H$-tilings will yield a monochromatic copy of~$mH$.
If~$G$ contains no such blow-up, we obtain enough structural information on the edge-colouring of~$G$ to conclude that a monochromatic copy of~$mH$ must exist. 

\smallskip

Before proceeding further, we define a special class of $(H,r)$-gadgets.

\begin{proposition}\label{prop:mean_value}
    Let $r\in\mathbb N$ and~$H$ be a graph.
    Let $F$ be a complete graph with vertex set $V(F)=[r]$ equipped with a loop at each vertex.
    Fix an~$r$-edge-colouring of~$F$ such that for every~$i\in[r]$,
    \begin{itemize}
        \item the loop on vertex~$i$ is coloured with~$i$ and 
        \item the set~$\text{Hom}(F_i\setminus \{i\}:H)$ is empty (recall~$F_i$ is the subgraph of~$F$ induced by the edges of colour~$i$).
    \end{itemize}
    Then $F$ is an $(H,r)$-gadget.
\end{proposition}

To prove Proposition~\ref{prop:mean_value} we use the Poincar{\'e}--Miranda theorem.

\begin{proof}[Proof of \cref{prop:mean_value}]
    For an~$r$-tuple ${\bf c}=\{c_v:v\in [r]\}$ with real entries between~$0$ and~$|H|$, let $f_i({\bf c}):=\nu_H(F_i,{\bf c})-|H|$ for every $i\in[r]$.
    Note that~$f_1,\dots,f_{r}$ are continuous functions\footnote{This follows, e.g., from Proposition~\ref{proposition:capacity_diff}.} with variables~$c_1,\dots,c_r\in[0,|H|]$.

    If $c_i=0$ then $\nu_H(F_i,{\bf c})=0$ by the second bullet point, and so $f_i({\bf c})=-|H|$.
    If $c_i=|H|$ then~$\nu_H(F_i,{\bf c})\ge|H|$, and so $f_i({\bf c})\ge 0$. 
    Therefore, we can apply the Poincar{\'e}--Miranda theorem and conclude that there exists a choice of~${\bf c}$ such that $f_i({\bf c})=0$ for every~$i\in[r]$, that is, $\nu_H(F_i,{\bf c})=|H|$ for every~$i\in[r]$.
\end{proof}

\begin{remark}
    If~$\chi(H)\ge r$, then (assuming the first bullet point holds) the second bullet point in \cref{prop:mean_value} is always satisfied: $F_i\setminus \{i\}$ is isomorphic to a subgraph of~$K_{r-1}$ and there is no graph homomorphism from~$H$ to~$K_{r-1}$.
\end{remark}

\begin{remark}\label{remark:chi>=r}
    Let~$F$ be as in \cref{prop:mean_value} where all non-loop edges are coloured with~$r$, and suppose $\chi(H)\ge r$.
    It is easy to see that there exists ${\bf c}=\{c_v:v\in [r]\}$ with $c_r>0$ and $c_i=|H|$ for every $i\in[r-1]$, such that $(F,{\bf c})$ is an $(H,r)$-gadget.
    By \cref{proposition:gadget_extremal}, it follows that $R_r(mH)\ge[(r-1)|H|+c_r]m-C_0$ for some constant~$C_0=C_0(H,r)$.
\end{remark}

\begin{remark}\label{remark:stronger_statement}
    It turns out that the statement of \cref{theorem:chi>=r'} still holds if we restrict to the $(H,r)$-gadgets as in \cref{prop:mean_value}.
    This will follow easily from the proof.
\end{remark}

We are now ready to prove \cref{theorem:chi>=r'}.

\begin{proof}[Proof of \cref{theorem:chi>=r'}]
    By \cref{lemma:gadget_induction}, for each $(H,r)$-gadget $F$ we can select a capacity~${\bf q_F}$ of rational entries such that $(F,{\bf q_F})$ is an $(H,r)$-gadget.
    Let $m_F\in\mathbb N$ be the integer output by \cref{proposition:gadget_induction} on input $(F,{\bf q_F})$; set $m_0:=\max\{m_F:F\text{ is an $(H,r)$-gadget}\}$.
    
    Pick integers $C\gg \ell\gg t\gg m_0, r,|H|$ such that $\ell$ is divisible by $|H|$. Recall that, by (\ref{eq:lower_bound}), to prove the theorem it suffices to
     prove that for every $m\in\mathbb N$,
    $$R_r(mH)\le\max\{||{\bf c}||:(F,{\bf c}) \text{ is an $(H,r)$-gadget}\}\cdot m+C.$$

    We proceed by induction on~$m$. 
    For the base case, the statement trivially holds for $m\le m_0$ since $C\gg m_0,r,|H|$.
    Thus, we may assume $m>m_0$.

    Let $G$ be an $r$-edge-coloured complete graph with
    $$|G|\ge\max\{||{\bf c}||:(F,{\bf c}) \text{ is an $(H,r)$-gadget}\}\cdot m+C.$$
    By \cref{remark:chi>=r}, we have $|G|\ge(r-1)|H|m+C/2$.
    
    Since $R_r(K_\ell)\le r^{r\ell}$~\cite{es}, there is a $K_\ell$-tiling $\mathcal K$ in $G$ consisting of monochromatic copies of $K_\ell$ which cover all but at most $r^{r\ell}$ vertices of $G$.

    \smallskip

    Suppose that $\mathcal K$ includes a copy of~$K_\ell$ in colour $i$ for every $i\in[r]$.
    Say $V_i$ is the vertex set of such a clique in colour $i$.
    Since $\ell\gg t$, by repeatedly applying the bipartite version of Ramsey's theorem, we can find a set $V_i'\subseteq V_i$ for each $i\in[r]$ such that
    \begin{itemize}
        \item $G[V_i',V_j']$ is a monochromatic complete bipartite graph for every $i,j\in[r]$ with $i\neq j$ and 
        \item $|V_i'|=t$ for every $i\in[r]$.
    \end{itemize}
    Let $B:=G[V_1'\cup\dots\cup V_r']$.
    Since $\chi(H)\ge r$ and $B[V_i']$ is coloured with~$i$ for every $i\in[r]$, by \cref{prop:mean_value} $B$ is a blow-up of an $(H,r)$-gadget $F$.
    As $t\geq m_0 \cdot |H|$,
    \cref{proposition:gadget_induction} implies  that~$B$ contains a blow-up~$B'$ of~$F$ such that $|B'|=m_F\cdot||{\bf q_F}||$ and~$B'$ contains a monochromatic copy of~$m_FH$ in each colour. 
    Let~$G':=G\setminus V(B')$.
    We have 
    $$|G'|=|G|-|B'|\ge\max\{||{\bf c}||:(F,{\bf c}) \text{ is an $(H,r)$-gadget}\}\cdot (m-m_F)+C.$$
    By the inductive hypothesis, $G'$ contains a monochromatic copy of $(m-m_F)H$, say in colour $i\in[r]$.
    Since $B'$ contains a monochromatic copy of $m_FH$ in colour $i$, we conclude that $G$ contains a monochromatic copy of $mH$, as required. 

    \smallskip

    Suppose instead that $\mathcal K$ does not include any copy of $K_\ell$ in colour $i$, for some $i\in[r]$.
   Then $\mathcal K$ contains at least 
    $$\frac{|G|-r^{r\ell}}{\ell(r-1)}\ge \frac{(r-1)|H|m+C/2-r^{r\ell}}{\ell(r-1)}\ge\frac{m|H|}{\ell}$$
    monochromatic copies of~$K_\ell$ in the same colour.
    Each such $K_\ell$ contains $\ell/|H|$ vertex-disjoint copies of $H$.
    Therefore, there is a monochromatic copy of~$mH$ in $G$, as required.
\end{proof}

\subsection{Proof of \cref{theorem:r=3'}}
At a high level, our proof strategy is  similar to that of \cref{theorem:chi>=r'} but with some relevant modifications.
Firstly, for the inductive step, we try to remove not only blow-ups of $(H,3)$-gadgets but also other types of structures.
In such structures, we do not insist that the largest tilings in each colour have approximately the same size (unlike blow-ups of $(H,3)$-gadgets) but every colour yields a tiling covering a proportion of the vertices {\it at least as large} as a monochromatic tiling in some blow-up of an $(H,3)$-gadget.
Secondly, the structural analysis is significantly more involved than the one for \cref{theorem:chi>=r'}.
We remark that if~$H$ is non-bipartite then $\chi(H)\ge 3$ and so the statement of the theorem follows already from \cref{theorem:chi>=r'}.

\smallskip

Similarly to the proof of \cref{theorem:chi>=r'}, we first define a special class of $(H,3)$-gadgets when~$H$ is bipartite. 

\begin{proposition}\label{proposition:gadgets_r=3}
    Let $H$ be a bipartite graph and let $F$ be a complete graph with vertex set $V(F)=\{1,2,3\}$ equipped with a loop at each vertex.
    Fix a $3$-edge-colouring of~$F$ such that
    \begin{itemize}
        \item the loop $(3,3)$ is coloured with $3$;
        \item $(1,2)$ and $(1,3)$ are coloured with $1$;
        \item $(2,3)$ is coloured with $2$;
        \item the loops $(1,1)$ and $(2,2)$ are not coloured with $3$.
    \end{itemize}
    Then $F$ is an $(H,3)$-gadget.
    Furthermore, if we instead colour $(i,i)$ and~$(i,i+1)$ with~$i$ for each $i\in[3]$ (entries are taken modulo~$3$), then~$F$ is also an $(H,3)$-gadget.
\end{proposition}
Recall that
given a graph $H$,  $\sigma(H)$ denotes the size of the smallest possible colour class in any proper $\chi(H)$-colouring of the vertices of $H$, whilst $\alpha(H)$ denotes the size of the largest independent set in $G$.
\begin{proof}[Proof of \cref{proposition:gadgets_r=3}]
    For the first part of the statement, we consider four possible cases for the colours assigned to the loops~$(1,1)$ and~$(2,2)$.
    \begin{itemize}
        \item $(1,1)$ and $(2,2)$ are coloured with $1$ and $2$ respectively.
        Let $${\bf c}=(c_1,c_2,c_3):=(|H|-\alpha(H),|H|-\alpha(H),|H|).$$
        It is easy to check that $\nu_H(F_i,{\bf c})=|F|$ for every~$i\in[3]$.
    
        \item $(1,1)$ and $(2,2)$ are both coloured with $1$.
        Let ${\bf c}=(c_1,c_2,c_3)$ with $0\le c_1\le|H|$, $c_2:=\sigma(H)$ and $c_3:=|H|$.
        It is easy to check that, for any value of $c_1$, we have $\nu_H(F_2,{\bf c})=\nu_H(F_3,{\bf c})=|H|$.
        Let $f(c_1):=\nu_H(F_1,{\bf c})-|H|$.
        Note that~$f$ is a continuous function.
        If $c_1=0$ then $\nu_H(F_{1},{\bf c})=\sigma(H)<|H|$ and so $f(c_1)<0$.
        If $c_1=|H|$ then~$\nu_H(F_{1},{\bf c})\ge|H|$, and so $f(c_1)\ge 0$. 
        By the Poincar{\'e}--Miranda theorem there exists a choice of~$c_1$ such that $f(c_1)=0$, that is, $\nu_H(F_1,{\bf c})=|H|$.
        
        \item $(1,1)$ and $(2,2)$ are both coloured with $2$.
        Let ${\bf c}=(c_1,c_2,c_3)$ with $c_1:=\sigma(H)$, $0\le c_2\le |H|$ and $c_3:=|H|$.
        It is easy to check that, for any value of $c_2$, we have $\nu_H(F_1,{\bf c})=\nu_H(F_3,{\bf c})=|H|$.
        Let $f(c_2):=\nu_H(F_2,{\bf c})-|H|$.
        If $c_2=0$ then $\nu_H(F_{2},{\bf c})=\sigma(H)<|H|$ and so $f(c_2)<0$.
        If $c_2=|H|$ then~$\nu_H(F_{2},{\bf c})\ge|H|$, and so $f(c_2)\ge 0$. 
        By the Poincar{\'e}--Miranda theorem there exists a choice of~$c_2$ such that $f(c_2)=0$, that is, $\nu_H(F_2,{\bf c})=|H|$.

        \item $(1,1)$ and $(2,2)$ are coloured with $2$ and $1$ respectively.
        Let ${\bf c}=(c_1,c_2,c_3)$ with $0\le c_1,c_2\le \sigma(H)$ and $c_3:=|H|$.
        For any values of $c_1$ and $c_2$, we have $\nu_H(F_3,{\bf c})=|H|$.
        Let $f_i(c_1,c_2):=\nu_H(F_i,{\bf c})-|H|$ for $i\in\{1,2\}$.
        Note that~$f_i$ is a continuous function for both $i\in\{1,2\}$.
        We have 
        $$f_1(0,c_2)=c_2-|H|\le\sigma(H)-|H|<0\qquad\text{and}\qquad f_1(\sigma(H),c_2)\ge 0.$$
        Similarly,
        $$f_2(c_1,0)=c_1-|H|\le\sigma(H)-|H|<0\qquad\text{and}\qquad f_2(c_1,\sigma(H))\ge 0.$$
        By the Poincar{\'e}--Miranda theorem there exists a choice of~$c_1$ and $c_2$ such that $f_1(c_1,c_2)=f_2(c_1,c_2)=0$, that is, $\nu_H(F_1,{\bf c})=\nu_H(F_2,{\bf c})=|H|$.
    \end{itemize}
    For the furthermore part, let ${\bf c}:=(|H|/2,|H|/2,|H|/2)$.
    It is then easy to check that~$(F,{\bf c})$ is an $(H,3)$-gadget.
\end{proof}

\begin{remark}\label{remark:r=3} 
    Let~$H$ be a bipartite graph.
    From the proof of the ``Furthermore" part of \cref{proposition:gadgets_r=3}, we immediately conclude that
    $$\max\{||{\bf c}||:(F,{\bf c}) \text{ is an $(H,3)$-gadget}\}\ge\frac{3}{2}\cdot|H|.$$
    Similarly, from  the case where both~$(1,1)$ and~$(2,2)$ are coloured~$1$, we conclude that
    $$\max\{||{\bf c}||:(F,{\bf c}) \text{ is an $(H,3)$-gadget}\}>|H|+\sigma(H).$$
    Note that $\sigma(H)\le |H|/2$ and so the second bound is better than the first only when $\sigma(H)=|H|/2$.
\end{remark}

We are ready to prove \cref{theorem:r=3'}.

\begin{proof}[Proof of \cref{theorem:r=3'}]
    If $H$ is non-bipartite, i.e., $\chi(H)\ge 3$, we are done by 
    \cref{theorem:chi>=r'}.
    Therefore, suppose $H$ is bipartite.

    \smallskip

    By \cref{lemma:gadget_induction}, for each $(H,3)$-gadget $F$ we can select a capacity ${\bf q_F}$ of rational entries such that $(F,{\bf q_F})$ is an $(H,3)$-gadget.
    Let $m_F\in\mathbb N$ be the integer output of \cref{proposition:gadget_induction} on input $(F,{\bf q_F})$; set $m_0:=\max\{m_F:F\text{ is an $(H,3)$-gadget}\}$.
    
    Define constants 
    \begin{align}\label{bighier}
    1/C\ll1/M\ll\eps\ll1/\ell_0\ll d\ll1/q_1 \ll 1/q\ll1/t\ll 1/|H|,1/m_0,
    \end{align}
    where~$C,M, \ell_0, q_1, q, t \in \mathbb N$  and~$q$ and~$t$ are divisible by $|H|$. By (\ref{eq:lower_bound}), to prove the theorem it suffices to
     prove that for every $m\in\mathbb N$,
    $$R_3(mH)\le\max\{||{\bf c}||:(F,{\bf c}) \text{ is an $(H,3)$-gadget}\}\cdot m+C.$$

    We proceed by induction on~$m$. 
    For the base case, the statement trivially holds for $m\le M$ since $C\gg M,|H|$.
    Thus, we may assume $m>M$.

    Let $G$ be a $3$-edge-coloured complete graph with
    $$|G|\ge\max\{||{\bf c}||:(F,{\bf c}) \text{ is an $(H,3)$-gadget}\}\cdot m+C.$$
    Note that, by \cref{remark:r=3}, we have $|G|\ge 3|H|m/2+C$.
    Furthermore, if $\sigma(H)=|H|/2$ then by \cref{remark:r=3} $|G|\ge (3|H|/2+c_0)m+C$ for some constant $c_0=c_0(H)>0$ independent of~$m$. 
    In particular, in this case
    we may assume that 
    \begin{align}\label{extra1}
      1/t \ll c_0.  
    \end{align}
    If $\sigma(H) < |H|/2$, then we may assume that
 \begin{align}\label{extra2}
      1/t \ll |H|/2-\sigma (H).  
    \end{align}

    We will frequently invoke the following simple claim.
    \begin{claim}\label{claim:inductive_step}
    If there exists a subgraph~$B$ of~$G$ such that, for some~$m'\in\mathbb N$,
    \begin{itemize}
        \item $B$ contains a monochromatic copy of~$m'H$ in each colour and
        \item $|B|\le\max\{||{\bf c}||:(F,{\bf c}) \text{ is an $(H,3)$-gadget}\}\cdot m'$,
    \end{itemize}
    then~$G$ contains a monochromatic copy of~$mH$.
    In particular, the second bullet point above holds if $|B|\le 3|H|m'/2$.
    \end{claim}

    \begin{proofclaim}
        Let~$G':=G\setminus V(B)$.
        We have
        $$|G'|=|G|-|B|\ge\max\{||{\bf c}||:(F,{\bf c}) \text{ is an $(H,3)$-gadget}\}\cdot (m-m')+C,$$
        hence by inductive hypothesis, $G'$ contains a monochromatic copy of~$(m-m')H$.
        Since~$B$ contains a monochromatic copy of~$m'H$ in the same colour, we conclude there is a monochromatic copy of~$mH$ in~$G$, as required.
    \end{proofclaim}

    Next, we show that if~$G$ contains a (sufficiently large) blow-up of a  rainbow triangle or a $K_4$ containing a spanning rainbow star, then we can find a suitable subgraph~$B$ to apply Claim~\ref{claim:inductive_step} to.
    
    \begin{claim}\label{claim:rainbow_triangle}
        If $G$ contains a blow-up~$B$ of a rainbow triangle with parts of size~$q$, then~$G$ contains a monochromatic~$mH$.
    \end{claim}

    \begin{proofclaim}
        Suppose~$B$ has parts~$V_1$, $V_2$ and~$V_3$, each of size~$q$.
        We may assume the edges in $B[V_1,V_2]$, $B[V_2,V_3]$ and~$B[V_1,V_3]$ are all coloured $1$, $2$ and~$3$, respectively.
        Since~$H$ is bipartite, for every $i\neq j$ the graph~$B[V_i,V_j]$ contains~$2q/|H|=:m'$ vertex-disjoint copies of~$H$.
        We conclude that~$B$ contains a monochromatic copy of~$m'H$ in each colour $i\in[3]$.
        Since $|B|=3q=(3/2)m'|H|$, we are done by \cref{claim:inductive_step}.
    \end{proofclaim}

    \begin{claim}\label{claimK}
    Let $K$ be a $3$-edge-coloured  $K_4$ 
        such that $K$ contains a spanning rainbow star. 
        If $G$ contains a blow-up~$B$ of $K$
        with parts of size~$q$, then~$G$ contains a monochromatic~$mH$.
    \end{claim}

    \begin{proofclaim}
        Suppose~$B$ has parts~$S_0$, $S_1$, $S_2$ and~$S_3$, each of size~$q$, where~$S_0$ corresponds to the center of the rainbow star.
        We may assume the edges in~$B[S_0,S_1]$, $B[S_0,S_2]$ and~$B[S_0,S_3]$ are all coloured $1$, $2$ and~$3$, respectively.

        By the previous claim we may assume $B[S_1\cup S_2 \cup S_3]$ is not the blow-up of a rainbow triangle.
        Hence, without loss of generality, we may assume that none of~$B[S_1,S_2]$, $B[S_2,S_3]$ and~$B[S_1,S_3]$ is coloured with~$3$ and at least two must be coloured with~$1$.
        However, if $B[S_2 , S_3]$ is coloured with~$1$ then~$B[S_0\cup S_2 \cup S_3]$ is a blow-up of a rainbow triangle, and we are done by the previous claim.
        Hence, we may assume~$B[S_1,S_2]$ and~$B[S_1,S_3]$ are coloured~$1$ and~$B[S_2,S_3]$ is coloured~$2$.

        Since~$q\gg t$, by Ramsey's theorem for every $i\in[3]\cup\{0\}$ we can find sets $S_i'\subseteq S_i$, each of size $2t$, such that
        $G[S'_i]$ is monochromatic.
        Select~$U_i\subseteq S_i'$ with $|U_i|=t$ for each $i\in[3]\cup\{0\}$.

         Suppose~$G[S'_1]$ or $G[S'_2]$ is coloured with~$3$.
        Let~$B_1:=G[U_0\cup U_3\cup S'_1\cup S'_2]$.
        Observe that~$B_1$ contains a monochromatic copy of $(4t/|H|)H$ in colours~$1$ and~$2$.
        On the other hand, $B_1[U_0,U_3]$ contains a monochromatic copy of $(2t/|H|)|H|$ in colour~$3$.
        If~$G[S'_1]$ (resp.\@ $G[S'_2]$) is coloured with~$3$ then it contains a monochromatic copy of $(2t/|H|)|H|$ in colour~$3$, and so~$B_1$ has a copy of $(4t/|H|)H$ in colour~$3$.
        Since $|B_1|=6t=3|H|m'/2$ with $m':=4t/|H|$, we are done by \cref{claim:inductive_step}.

        If both~$G[S_0']$ and~$G[S_3']$ are coloured~$1$, then let~$B_2:=G[S_0'\cup S_2'\cup S_3']$.
        Then~$|B_2|=6t$ and~$B_2$ contains a monochromatic copy of $(4t/|H|)H$ in each colour $i\in[3]$, so we are done by \cref{claim:inductive_step}.
        
        Similarly, if both~$G[S_0']$ and~$G[S_3']$ are coloured~$2$, let~$B_3:=G[S_0'\cup S_1'\cup S_3']$.
        Note~$|B_3|=6t$ and~$B_3$ contains a monochromatic copy of $(4t/|H|)H$ in each colour $i\in[3]$, so we are done by \cref{claim:inductive_step}.

        If~$G[S_0']$ and~$G[S_3']$ are coloured with~$1$ and~$2$ respectively (or vice versa), then let~$B_4:=G[S_0'\cup S_3'\cup U_1\cup U_2]$.
        Note~$|B_4|=6t$ and~$B_4$ contains a monochromatic copy of $(4t/|H|)H$ in each colour $i\in[3]$, so we are done by \cref{claim:inductive_step}

        Therefore,  we may assume (i)~$G[S_1']$ and~$G[S_2']$ are not coloured with~$3$ and (ii)~$G[S_0']$ or $G[S_3']$ is coloured with~$3$.
        Assume that $G[S_3']$ is coloured with~$3$; the case when $G[S_0']$
        is coloured with~$3$ follows analogously.
        Let $B_5:=G[S_1'\cup S_2'\cup S_3']$.
        To recap, we have that~$G[S_3']$, $G[S_1',S_2']$, $G[S_1',S_3']$ and~$G[S_2',S_3']$ are coloured with~$3$, $1$, $1$ and~$2$, respectively; furthermore, $G[S_1']$ and~$G[S_2']$ are not coloured with~$3$.

        By \cref{proposition:gadgets_r=3}, we conclude~$B_5$ is the blow-up of an~$(H,3)$-gadget $(F,{\bf q_F})$.
        Since~$|S'_i| =2 t\gg m_0, |H|$, by \cref{proposition:gadget_induction} $B_5$ contains a blow-up $B'_5$ of $F$ such that $|B'_5|=m_F\cdot||{\bf q_F}||$ and $B'_5$ contains a monochromatic copy of $m_FH$ in each colour. 
        Thus, we are done by \cref{claim:inductive_step}. 
    \end{proofclaim}

    By the previous claims, we may assume that~$G$ does not contain a blow-up of a rainbow star  of size~$q$ or a blow-up of a $K_4$ as in Claim~\ref{claimK}.    
    We will use this structural information to find a monochromatic copy of~$mH$.
In particular, under these assumptions we will not need to apply the induction hypothesis to obtain this monochromatic $mH$.

    \medskip

Note that $|G| \geq M$ as $m > M$. Thus, we can
   apply \cref{theorem:regularity} to~$G$ with parameters $r:=3,\ell_0 \in \mathbb N$ and $0<\eps,d<1$ to obtain a partition $V_0,V_1, \dots,V_\ell$ with $\ell_0\le \ell\le M$ (recall we picked~$M$ so that $1/M\ll \eps,1/\ell_0$), a spanning subgraph~$G'$ of~$G$ and a reduced graph~$R$.
    Note that $|V_0|\le\eps |G|$ and so $|V_i|=(|G|-|V_0|)/\ell\ge(1-\eps)|G|/M$ for every $i\in[\ell]$.
 As $\eps \ll d$, \cref{reducedgraph} implies that  $\delta(R)\ge(1-6d)|R|$.

    Recall that~$R$ is $3$-edge-coloured and has vertex set $\{V_1,\dots,V_\ell\}$. Further,   if $V_xV_y\in E(R)$ is coloured $c \in [3]$, then $G'_c[V_x,V_y]$ is $[\eps,d]$-regular.

    \begin{claim}\label{claimR}
    $R$ contains no rainbow star on $4$ vertices nor a rainbow triangle.
    \end{claim}
 \begin{proofclaim}
If $R$ contains a rainbow triangle, then  by \cref{corollary:embedding}
$G' \subseteq G$ contains a blow-up of a rainbow triangle with parts of size $q$, a contradiction to our assumption.

If $R$ contains a rainbow star $S$ on $4$ vertices, then  by \cref{corollary:embedding}
$G' \subseteq G$ contains a blow-up of  $S$ with parts of size $q_1$.
As $1/q_1 \ll 1/q$ and $G$ is a complete graph, we may apply the bipartite version of
Ramsey's theorem to thus find a blow-up of a $3$-edge-coloured $K_4$ as in Claim~\ref{claimK}, a contradiction to our assumption.
\end{proofclaim}

    Claim~\ref{claimR} implies that
     there exists a partition $V(R)=W_{12}\cup W_{23}\cup W_{13}$ where the vertices in~$W_{ij}$ are incident to edges of colour~$i$ or~$j$ in $R$ but not of colour~$k$, for any permutation $ijk$ of $123$.
    Note that all edges between $W_{ij}$ and $W_{jk}$ in $R$ must be coloured~$j$.

    Without loss of generality, we assume that $|W_{12}|\ge|W_{13}|\ge|W_{23}|$.
    Since $\delta(R)\ge(1-6d)|R|$, if $|W_{23}|>6d|R|$ then there is an edge~$V_xV_y \in E(R)$ between~$W_{13}$ and~$W_{23}$.
    Since $|W_{12}|\ge |R|/3>12d|R|$, we know~$V_x$ and~$V_y$ have a common neighbour~$V_z$ in~$W_{12}$.
    But then $V_xV_yV_z$ is a rainbow triangle in $R$, a contradiction.
    Thus, we may assume $|W_{23}|\le 6d|R|$. 
    
    Observe that the subgraph $R[W_{12},W_{13}]$ is close to being complete bipartite and all its edges are coloured~$1$.
    In the next two claims, we exploit these properties to find $H$-tilings in $R[W_{12},W_{13}]$ coloured with~$1$.
   \begin{claim}\label{claim:bipartite_embedding}
         For any subsets $A\subseteq W_{13}$ and $B\subseteq W_{12}$ with $|A|=|B|\ge 9d|R|$, there exists a copy of~$H$ in $R[A,B]$ with precisely~$\sigma(H)$ vertices in~$A$ and $|H|-\sigma(H)$ vertices in~$B$, or vice versa.
    \end{claim}

    \begin{proofclaim}
        Since $\delta(R)\ge(1-6d)|R|$, we have 
        $$\delta(R[A,B])\ge|A|-6d|R|\ge |A|/3$$
        and so $e(R[A,B])\ge |A||B|/3$.
        Since $|A|\ge 9d|R|$  by, e.g., the K\"ov\'ari--S\'os--Tur\'an theorem there exists a copy of~$H$ in $R[A,B]$ with precisely~$\sigma(H)$ vertices in~$A$ and $|H|-\sigma(H)$ vertices in~$B$ (or vice versa).      
    \end{proofclaim}

    \begin{claim}\label{claim:bipartite_embedding+}
         For any subsets $A\subseteq W_{13}$ and $B\subseteq W_{12}$ with 
         $$\min\left\{\frac{|A|}{|B|},\frac{|B|}{|A|}\right\}\ge\frac{\sigma(H)}{|H|-\sigma(H)},$$ 
         there exists an~$H$-tiling in $R[A,B]$ covering all but at most~$18d|H||R|$ vertices of~$R[A,B]$.
    \end{claim}

    \begin{proofclaim}
        Without loss of generality, suppose that $|A|\le |B|$.
        Using \cref{claim:bipartite_embedding},
        we repeatedly remove copies of~$H$ from~$R[A,B]$ with~$\sigma(H)$ vertices in~$A$ and $|H|-\sigma(H)$ vertices in~$B$  until we either cover all but at most
        $9d|R|$ vertices in~$A$ or the number of vertices left uncovered in~$A$ and~$B$ differ by at most~$|H|$.
        Note that one of these two instances must occur since $|A|\le|B|$, $\sigma(H)\le|H|-\sigma(H)$ and, by \cref{claim:bipartite_embedding}, we can find such a copy of~$H$ whenever there are at least~$9d|R|$ vertices left uncovered in both~$A$ and~$B$.

        Suppose we stopped when all but~$9d|R|$ vertices in~$A$ are covered.
        We thus have an $H$-tiling covering at least $|A|-9d|R|$ vertices in~$A$.
        Hence, the number of vertices covered in~$B$ is at least
        $$\frac{|H|-\sigma(H)}{\sigma(H)}\cdot(|A|-9d|R|)\ge|B|-9d|H||R|,$$
        where here we use the inequality in the statement of the claim and $(|H|-\sigma(H))/\sigma(H)\le|H|$.
        This $H$-tiling satisfies the required properties.

        Suppose instead we stopped when the number of vertices left uncovered in~$A$ and~$B$ differ by at most~$|H|$.
        We now keep removing copies of~$H$ as follows.
        We remove two vertex-disjoint copies of~$H$ at a time, one with~$\sigma(H)$ and~$|H|-\sigma(H)$ vertices in~$A$ and~$B$ respectively, and the other with~$\sigma(H)$ and~$|H|-\sigma(H)$ vertices in~$B$ and~$A$ respectively.
        Observe that the difference between the uncovered vertices in $A$ and $B$ does not change and so it is always at most $|H|$.
        By \cref{claim:bipartite_embedding}, we can repeat this until the number of vertices left uncovered in~$A\cup B$ is at most~$18d|R|+2|H|\le18d|H||R|$.
        Combining these copies of $H$ with our initial $H$-tiling yields the desired $H$-tiling in $R[A,B]$.
    \end{proofclaim}

We now split into cases depending on the size of $|W_{13}|$.

    \noindent{ \bf Case 1:}
    $|W_{13}|\ge(1+\sqrt{d})\frac{2\sigma(H)\cdot |R|}{3|H|}.$
    
    Let $S$ be a subset of $W_{12}$ with $|S|=\min\{|W_{12}|,\lfloor (|H|-\sigma(H))|W_{13}|/\sigma(H) \rfloor \}$.
    Note that $1\ge|W_{13}|/|S|\ge\sigma(H)/(|H|-\sigma(H))$.
    Taking $A:=W_{13}$ and $B:=S$, 
     \cref{claim:bipartite_embedding+} implies that there exists an $H$-tiling in $R[W_{12},W_{13}]$ covering at least $|A|+|B|-18d|H||R|$ vertices; so at least
    $$\min\left\{|W_{12}|+|W_{13}|,\frac{|H|}{\sigma(H)}\cdot|W_{13}|-1\right\}-18d|H||R|\ge\frac{(1+d)2|R|}{3}$$ vertices.
    In particular, this is a monochromatic $H$-tiling in $R$ in colour~$1$.
    By \cref{corollary:embedding}, $G$ contains a monochromatic $H$-tiling in colour~$1$ covering at least
    \begin{align*}
        (1-\sqrt\eps)\cdot\frac{2(1+d)|R|}{3}\cdot\frac{|G|-|V_0|}{|R|}&\ge\frac{2(1-\sqrt\eps)(1+d)(1-\eps)|G|}{3}\ge\frac{2|G|}{3}
    \end{align*}
    vertices in~$G$, where the last inequality follows as $\eps \ll d$.
    Since $|G|\ge3m|H|/2$ by Remark~\ref{remark:r=3}, it follows that we have found a monochromatic copy of~$mH$, as desired.

\smallskip

    \noindent{ \bf Case 2:} $|W_{13}|\le |R|/q$.
    
    In this case we have $$\delta(R[W_{12}])\ge\delta(R)-|W_{13}|-|W_{23}|\ge(1-6d-1/q-6d)|R|\stackrel{(\ref{bighier})}{\ge}(1-2/q)|W_{12}|.$$
    Recall that $R[W_{12}]$ is $2$-edge-coloured and note that $1/|W_{12}| \ll 1/q \ll 1/t$. Thus, we can apply
    Lemma~\ref{robust2colours} with $|W_{12}|,2/q,1/(2t)$ playing the role of $n,\delta,\eta$, respectively, to obtain a monochromatic $H$-tiling  in $R[W_{12}]$ covering at least
    $$
    \frac{|W_{12}||H|}{2|H|-\alpha(H)} - \frac{|W_{12}|}{2t} \geq
    \frac{(1-1/t)|W_{12}||H|}{2|H|-\alpha(H)}\ge\frac{(1-1/t)(1-2/q)|R||H|}{|H|+\sigma(H)}
    $$
    vertices; here we used the inequality $\alpha(H)\ge|H|-\sigma(H)$.
    By \cref{corollary:embedding}, there exists a monochromatic $H$-tiling in~$G$ covering at least
    $$\frac{(1-\sqrt{\eps})(1-1/t)(1-2/q)|H|(|G|-|V_0|)}{|H|+\sigma(H)}
    \stackrel{(\ref{bighier})}{\ge}
    \frac{(1-3/t)|H||G|}{|H|+\sigma(H)}$$
    vertices of~$G$.
    The right-hand side of the last inequality is always at least~$m|H|$, as desired.
    Indeed, if $\sigma(H)<|H|/2$ then this follows from (\ref{extra2}) and $|G|\ge 3m|H|/2$; if $\sigma(H)=|H|/2$ then this follows from (\ref{extra1}) and $|G|\ge (3|H|+c_0)m/2$.

\smallskip

\noindent{ \bf Case 3:}
    $\frac{|R|}{q}\le|W_{13}|\le(1+\sqrt d)\frac{2\sigma(H)|R|}{3|H|}.$
    
    We deal with this final case  by splitting into two further subcases.

    \smallskip

    \noindent{\it Case 3a:} $\sigma(H)<|H|/2$. 
    We have 
    \begin{align}\label{newX}
    |W_{13}|\le(1+\sqrt d) \frac{2\sigma(H)|R|}{3|H|}\le(1+\sqrt d) \frac{(|H|-1)|R|}{3|H|}\le\frac{(1-\sqrt d)|R|}{3},
    \end{align}
    where the last inequality holds since $d\ll 1/|H|$.

    Let $m_1:=\lceil (1+d)|R|/3\rceil$, $m_2:=|W_{12}|-2m_1$ and $s:=\lceil |R|/4 \rceil$.
    Note that
    \begin{align*}
    m_2=|W_{12}|-2m_1=(|R|-|W_{13}|-|W_{23}|)-2m_1&\stackrel{(\ref{newX})}{\ge}\left(1-\frac{1-\sqrt{d}}{3}-6d-\frac{2(1+d)}{3}\right)|R|-2\ge 1.
    \end{align*}
    Furthermore,
    $$m_2\le |R|-2m_1\le\frac{(1-2d)|R|}{3}\le m_1.$$
    Let~$R'$ be the $3$-edge-coloured complete graph on vertex set~$W_{12}$ obtained from~$R[W_{12}]$ by adding each missing edge between distinct vertices and colouring them with~$3$.    
    Since $|W_{12}|=2m_1+m_2$ and $m_1\ge m_2, s$, by \cref{theorem:Gyarfas_Sarkozy} $R'$ must contain a copy of $m_1K_2$ in colour~$2$, a copy of $m_2K_2$ in colour~$1$ or a copy of~the star $S_{s}$ in colour~$3$.

    The last case is not possible since $\delta(R)\ge(1-6d)|R|$ implies each vertex in~$W_{12}$ is incident to at most $6d|R|$ missing edges in $R$ and so $R'$ does not contain a star $S_{s}$ in colour $3$.
    
    If~$R'$ contains a copy of~$m_1K_2$ in colour~$2$, then by \cref{corollary:embedding} $G$ contains a monochromatic $H$-tiling in colour~$2$ covering at least 
    $(2/3)|G| \geq m|H|$ vertices of~$G$, as required.

    Finally, suppose that $R'$ contains a copy $\mathcal M$ of~$m_2K_2$ in colour~$1$.
    The number of uncovered vertices in $W_{12}$ is 
    $$|W_{12}|-2m_2=4m_1-|W_{12}|\geq \frac{4(1+d)}{3}|R|-|W_{12}|\ge\frac{1+4d}{3}|R|+|W_{13}|.$$
    Let $A$ be a subset of these uncovered vertices with 
    $$|A|=\min \left \{\lfloor (|H|-\sigma(H))|W_{13}|/\sigma(H) \rfloor \ ,\left \lfloor \frac{1+4d}{3}|R|+|W_{13}| \right \rfloor \right \},$$
    and set $B:=W_{13}$.
    By \cref{claim:bipartite_embedding+}, there exist an $H$-tiling $\mathcal H$ in $R[A,B]$ covering at least $|A|+|B|-18d|H||R|$ vertices in $R[W_{12},W_{13}]$;  so at least
    $$\min\left\{\frac{1+4d}{3}|R|+2|W_{13}|,\frac{|H||W_{13}|}{\sigma(H)}\right\}-1-18d|H||R|\ge\left(2+\sqrt{d}\right)|W_{13}|$$
    vertices (here the last inequality follows as $\sigma (H) \leq (|H|-1)/2$).
    This $H$-tiling $\mathcal H$ is monochromatic in colour~$1$ and is vertex-disjoint from $\mathcal M$.
    
    Together $\mathcal M$ and $\mathcal H$ cover at least
    \begin{align*}
2m_2+\left(2+\sqrt{d}\right)|W_{13}| &\geq 
    2|W_{12}|-\frac{4(1+d)|R|}{3}-4+\left(2+\sqrt{d}\right)|W_{13}|\\ &=\frac{2|R|}{3}-2|W_{23}|-\frac{4d|R|}{3}+\sqrt{d}|W_{13}|-4
        \ge\frac{(1+d)2|R|}{3}
    \end{align*}
    vertices in~$R$, where in the last inequality we use that $|W_{13}|\ge|R|/q$; $|W_{23}|\le 6d|R|$; $d\ll 1/q$.
    By applying Lemma~\ref{corollary:embedding} twice (once for $\mathcal M$ and once for $\mathcal H$), we obtain a monochromatic $H$-tiling in colour $1$ in $G$
    covering at least
    $$(1-\sqrt{\eps}) \cdot \frac{(1+d)2|R|}{3} \cdot \frac{(|G|-|V_0|)}{|R|} 
    \stackrel{(\ref{bighier})}{\geq} \frac{2|G|}{3} \geq m|H|$$
    vertices, as desired.

    \smallskip
    
    \noindent{\it Case 3b:} $\sigma(H)=|H|/2$.
    We have
    \begin{align}\label{extraD}
    |W_{13}|+6d|R|\le(1+\sqrt d)\frac{2\sigma(H)|R|}{3|H|}+6d|R|<(1+3\sqrt d)\frac{|R|}{3}.
    \end{align}
    Let $m_1:=s:=\lceil (1+3\sqrt{d})  |R|/3\rceil$.
    Let~$R'$ be the complete graph obtained from~$R[W_{12}\cup W_{13}]$ by adding each missing edge between distinct vertices and colouring them with~$3$, and adding a set~$S$ of new vertices with~$|S|=3m_1-|W_{12}|-|W_{13}|>0$ and colour all edges incident to them with~$1$ or~$2$.
    
    By \cref{theorem:Gyarfas_Sarkozy}, $R'$ contains a copy of $m_1K_2$ in colour~$1$ or~$2$, or a copy of~$S_s$ in colour~$3$.
    However, note that~$R'$ does not contain a copy of~$S_s$ in colour~$3$, since a vertex in~$R'$ is incident to at most $$|W_{13}|+6d|R|\stackrel{(\ref{extraD})}{<} s$$ edges coloured with~$3$.

    Hence, $R'$ contains a copy of $m_1K_2$ in, say,  colour $1$; so $R$ contains a copy $\mathcal K$ of $(m_1-|S|)K_2$ in colour~$1$.
    Notice $|S| \leq |R| + 4 \sqrt {d} |R|/3 -|W_{12}|-|W_{13}| \leq  4 \sqrt {d} |R|/3+ |W_{23}| 
    \leq 5 \sqrt {d} |R|/3$.
    Thus, $\mathcal K$ covers at least
    $$2(m_1-|S|)\ge\frac{2(1+3\sqrt{d})|R|}{3}-\frac{10\sqrt{d}|R|}{3}\ge\frac{2(1-2\sqrt{d})|R|}{3}$$
    vertices of $R$.
    It follows from \cref{corollary:embedding} that $G$ contains an $H$-tiling in colour $1$ covering at least
    $$\frac{2(1-\sqrt{\eps})(1-2\sqrt{d})(|G|-|V_0|)}{3}
    \stackrel{(\ref{bighier})}{\ge}\frac{2(1-3\sqrt{d})|G|}{3}$$
    vertices of~$G$.
    Since $|G|\ge(3|H|/2+c_0)m$ and $d \ll c_0$ by (\ref{bighier}) and (\ref{extra1}), we have found a monochromatic copy of~$mH$, as required.

\end{proof}

\section{Proof of \cref{thm:bi}}\label{sec:bip}

If~$H$ is a complete bipartite graph, we are able to determine~$R_r(mH)$ (up to a linear error term) as a function of a parameter independent of~$m$.
In fact, we prove a more general asymmetric result that also holds for some non-complete bipartite graphs.
Given $r \in \mathbb N$ and graphs $F_1,\dots, F_r$ we write 
$R_r(F_1,\dots, F_r)$ for the smallest $n \in \mathbb N$ such that every $r$-edge-colouring of $K_n$ contains a copy of $F_i$ in colour $i$ for some $i \in [r]$.
We will use the notation $R_r(m_iH_i)$ to indicate the Ramsey number $R_r(m_1H_1,\ldots,m_rH_r)$. 
Furthermore, when $m_1=\cdots=m_r=1$, we will use $R_r(H_i)$ to indicate the Ramsey number $R_r(H_1,\ldots,H_r)$. 

For a graph~$H$, we let 
$$s(H):=\max\left\{\frac{|I|}{|N_H(I)|}:\text{ $I\subseteq V(H)$ is an independent set}\right\},$$
where~$N_H(I)$ is the set of vertices in $V(H)\setminus I$ which are neighbours of some vertex in~$I$.
Note that, if~$H$ is complete bipartite, then $s(H)=(|H|-\sigma(H))/\sigma(H)$.
Next, we introduce the key parameter that governs~$R_r(mH)$ when $H$ is complete bipartite.
This parameter is well-defined for arbitrary graphs.

\begin{define}[$\mathcal F_r(m_iH_i)$ and~$t_r(m_iH_i)$]\label{define:lexico_cliquecover_parameter}
    Let $r\in\mathbb N$ with $r\ge2$.
    Let $H_1,\dots,H_r$ be graphs.
    Let $m_1,\dots,m_r\ge 1$.\footnote{Observe that we do not require that $m_1,\dots,m_r$ are integers.}
    Let~$\mathcal F_r(m_iH_i)$ denote the family of the following $r$-edge-coloured weighted graphs~$(F,{\bf c})$.   
    Let~$F$ be the complete graph on vertex set $V(F)=\{1,2,\dots, r\}\cup 2^{[r]}$ equipped with a loop at each vertex.  
    Consider non-negative weights ${\bf c}=\{c_v:v\in V(F)\}$ such that
    \begin{itemize}
        \item for every $i\in[r]$ we have
        $$(s(H_i)+1)\cdot c_{i}+\sum_{i\in I\subseteq[r]}c_{I}=m_i|H_i|,$$
        \item if $I,I'\subseteq[r]$ and $I\cap I'=\emptyset$ then $c_I=0$ or $c_{I'}=0$.
    \end{itemize}
    Pick an $r$-edge-colouring of~$F$ such that if an edge~$e$ is coloured with~$i$ then one of the following holds: $e$ is incident to the vertex~$i$, or $e=I_1I_2$ for some $I_1,I_2\subseteq [r]$ with $i\in I_1\cap I_2$, or~$e$ is incident to a vertex with weight~$0$.
    Let $t_r(m_iH_i):=\max\{||{\bf c}||:(F,{\bf c})\in\mathcal F_r(m_iH_i)\}$.
\end{define}

\begin{remark}
    It is easy to check that for any $r\in\mathbb N$ with $r\ge2$, graphs $H_1,\dots,H_r$ and $m\ge1$, we have $t_r(mH_i)=m\cdot t_r(H_i)$.
\end{remark}

In the next proposition, we show how blowing up members of the family~$\mathcal F_r(m_iH_i)$ yields complete $r$-edge-coloured graphs which do not contain a monochromatic copy of~$m_iH_i$ in colour~$i$, for every $i\in[r]$.
This provides a general lower bound for the Ramsey number~$R_r(m_iH_i)$.

\begin{proposition}\label{proposition:blow_up}
    Let $r\in\mathbb N$ with $r\ge2$.
    Let $H_1,\dots,H_r$ be graphs.
    Let $m_1,\dots,m_r\in\mathbb N$. 
    Let $(F,{\bf c})\in\mathcal F_r(m_iH_i)$.
    If~$B$ is a blow-up of~$F$ where each vertex $v\in V(F)$ is replaced by a class of $\left\lceil c_v-1\right\rceil$ vertices, then~$B$ does not contain a monochromatic copy of~$m_iH_i$ in colour~$i$ for every $i\in[r]$, and $|B|\ge ||{\bf c}||-2^r-r$.
    In particular, we have $R_r(m_iH_i)> t_r(m_iH_i)-2^r-r$.
\end{proposition}

\begin{proof}
    Let~$X_1,\dots,X_r\subseteq V(B)$ be the vertex classes corresponding to the vertices $1,2,\dots, r\in V(F)$.
    Furthermore, for every $i\in[r]$ let~$Y_i$ be the union of all vertex classes corresponding to some vertex~$v_I\in V(F)$ with~$i\in I$.
    By \cref{define:lexico_cliquecover_parameter}, for every $i\in[r]$ we have that an edge coloured~$i$ is either incident to~$X_i$ or lies in~$Y_i$.\footnote{Recall from \cref{define:lexico_cliquecover_parameter} that an edge in~$F$ coloured with~$i$ can be incident to a vertex of weight~$0$.
    Such a vertex is removed in the blow-up~$B$, hence all edges incident to it are removed too.}

    Say~$T$ is a monochromatic copy of~$H_i$ in~$B$ of colour~$i$.
    There is no edge of colour~$i$ in~$V(B)\setminus(X_i\cup Y_i)$, hence $V(T)\setminus(X_i\cup Y_i)$ is an independent set in~$T$.
    Furthermore, since there are no edges of colour~$i$ between~$Y_i$ and~$V(B)\setminus(X_i\cup Y_i)$, it follows that $N_T(V(T)\setminus(X_i\cup Y_i))\subseteq X_i$. 
    Hence
    $$\frac{|V(T)\setminus(X_i\cup Y_i)|}{|V(T)\cap X_i|}\le s(H_i).$$


    We conclude that a monochromatic copy of~$m_iH_i$ in~$B$ of colour~$i$ has at most~$s(H_i)\cdot|X_i|$ vertices in~$V(B)\setminus(X_i\cup Y_i)$.
    It follows that
    \begin{align*}
    m_i|H_i|&\le |X_i|+|Y_i|+s(H_i)\cdot|X_i|<(s(H_i)+1)\cdot c_i+\sum_{i\in I\subseteq[r]}c_I = m_i|H_i|,
    \end{align*}
    a contradiction.
    Thus~$B$ does not contain a monochromatic copy of~$m_iH_i$ in colour~$i$, for every $i\in[r]$.
    Finally, we have 
    $$|B|=\sum_{v\in V(F)} \lceil c_v-1\rceil\ge\sum_{v\in V(F)}(c_v-1)=||{\bf c}||-2^r-r.$$
\end{proof}

\begin{remark}\label{remark_bip}
    Observe that, when $H$ is complete bipartite and $H_i=H$ and $m_i=m$ for every $i\in[r]$, the graph $B$ and the sets $X_1,\dots,X_r$ and $Y_1,\dots,Y_r$ in the proof of \cref{proposition:blow_up} satisfy the assumptions of \cref{definition:extremal_construction_bi}, since $s(H)=(|H|-\sigma(H))/\sigma(H)$.
    Therefore, the proof of \cref{proposition:blow_up} yields that $b_r(m,H)\ge t_r(mH)-2^r-r$.
\end{remark}

The lower bound in \cref{proposition:blow_up} is not tight in general.
However, in the next definition, we introduce a special class of bipartite graphs for which this bound is attained, up to a linear error term.
This class includes complete bipartite graphs.

\begin{define}[Edge-contractible]
    A graph~$H$ is {\it edge-contractible} if there exists a fractional $H$-tiling $w:\text{Hom}(H:H)\to\mathbb R$ of~$H$ such that~$w(H)=|H|$ and, for every $\phi\in\text{Hom}(H:H)$, we have $w(\phi)>0$ if and only if the image of~$\phi$ has size~$2$, that is, $\phi$ sends~$H$ to an edge.
    In particular, an edge-contractible graph must be bipartite.
\end{define}

\begin{remark}
    If~$H$ is complete bipartite then it is edge-contractible.
    Indeed, say the bipartite parts of~$H$ are~$A$ and~$B$.
    For every $a\in A$ and $b\in B$ let $\phi_{ab}\in\text{Hom}(H:H)$ with
    $$\phi_{ab}(v)
    =\begin{cases}
    a \text{ if $v\in A$},\\
    b \text{ if $v\in B$}.
    \end{cases}$$
    Let $w(\phi_{ab})=(|A|\cdot|B|)^{-1}$ for every $a\in A$ and $b\in B$, and $w(\phi)=0$ for all other homomorphisms~$\phi\in\text{Hom}(H:H)$.
    Then $w$ is a fractional $H$-tiling of~$H$ with~$w(H)=|H|$.
    Similarly, it is easy to see that the vertex-disjoint union of multiple copies of $H$ is also edge-contractible.
\end{remark}

\begin{remark}
    If~$H$ is edge-contractible then $s(H)=(|H|-\sigma(H))/\sigma(H)$.
    Indeed, let $I$ be an independent set of $H$ and let $w$ be a fractional $H$-tiling of~$H$ such that~$w(H)=|H|$ and $w(\phi)>0$ if and only if the image of~$\phi$ has size~$2$.
    Any such homomorphism whose image intersects~$I$ must satisfy $\text{Im}(\phi)=\{x,y\}$ for some $x\in I$ and $y\in N_H(I)$.
    Since $\phi^{-1}(x)$ and $\phi^{-1}(y)$ are both independent sets in $H$, we have
    $$\sigma(H)\le|\phi^{-1}(x)|,|\phi^{-1}(y)|\le|H|-\sigma(H)$$ 
    and so
    \begin{align*}
    |I|&=\sum \{w(\phi)\cdot|\phi^{-1}(x)|:x\in I, y\in N_H(I),\text{Im}(\phi)=\{x,y\}\}\\
    &\le\frac{|H|-\sigma(H)}{\sigma(H)}\cdot\sum \{w(\phi)\cdot|\phi^{-1}(y)|:x\in I, y\in N_H(I),\text{Im}(\phi)=\{x,y\}\} \\       
    &\le \frac{|H|-\sigma(H)}{\sigma(H)}\cdot|N_H(I)|.  
    \end{align*}
    As $I$ is arbitrary, it follows that $s(H)=(|H|-\sigma(H))/\sigma(H)$ as required.
\end{remark}

The next theorem states that~$R_r(mH_i)$ equals $m\cdot t_r(H_i)$, up to a linear error term, for any edge-contractible graphs~$H_1,\dots,H_r$ and $m\in\mathbb N$.
In particular, this holds when $H_i=H$ for all $i\in[r]$ and $H$ is a complete bipartite graph.

\begin{theorem}\label{theorem:complete_bipartite}
    Let $r\in\mathbb N$ with $r\ge2$.
    Let $H_1,\dots,H_r$ be edge-contractible graphs.
    For every~$\eps>0$, there exists~$m_0\ge1$ such that for all integers $m\ge m_0$ we have 
    $$m\cdot t_r(H_i)-2^r-r\le R_r(m H_i)\le(1+\eps)\cdot m\cdot t_r(H_i).$$
\end{theorem}

To see why \cref{theorem:complete_bipartite} implies \cref{thm:bi}, consider the following.
As observed in \cref{remark_bip}, we have $R_r(mH)\ge b_r(m,H)\ge t_r(mH)-2^r-r$ for every complete bipartite graph $H$ and $m\in\mathbb N$.
This inequality together with \cref{theorem:complete_bipartite} imply that~$b_r(m,H)$ and $R_r(mH)$ differ by a linear error term, which is precisely the statement of \cref{thm:bi}.

Note that, by letting~$H_i:=q_iF_i$ with~$q_i\in\mathbb N$ and~$F_i$ complete bipartite, the above theorem determines the asymmetric Ramsey number~$R_r((mq_i)F_i)$ up to a linear error term, provided $m$ is sufficiently large.

The rest of the section covers the proof of Theorem~\ref{theorem:complete_bipartite}.



\subsection{Deducing Theorem~\ref{theorem:complete_bipartite} from a fractional analogue}

The following theorem is a fractional dense analogue of \cref{theorem:complete_bipartite}.
Combining it with the Regularity Lemma immediately yields \cref{theorem:complete_bipartite}.

\begin{theorem}[Dense fractional version]\label{theorem:fractional_bipartite}
    Let $r\in\mathbb N$ with $r\ge2$.
    Let $H_1,\dots,H_r$ be edge-contractible graphs.
    For every~$\eps>0$ there exists $m_0\ge1$ such that the following holds.
    If~$R$ is an $r$-edge-coloured graph with $\delta(R)\ge(1-2r\eps)|R|$ and~$|R|\ge m_0$, then 
    $$\nu_{H_i}(R_i)\ge(1-\eps2^{3r})\cdot\frac{|H_i|}{t_r(H_i)}\cdot|R|$$ 
    for some $i\in[r]$, where~$R_i$ denotes the subgraph of~$R$ spanned by the edges coloured~$i$.
\end{theorem}

\begin{proof}[Proof of~\cref{theorem:complete_bipartite} via Theorem~\ref{theorem:fractional_bipartite}]
    
    By \cref{proposition:blow_up} we have $R_r(mH_i)\ge t_r(mH_i)-2^r-r=m\cdot t_r(H_i)-2^r-r$ for any~$m\in\mathbb N$.
    Next, we prove the upper bound.

    Pick constants
    $$1/M\ll\eps'\ll1/\ell_0\ll d\ll\eps\ll 1/|H_1|,\dots,1/|H_r|,1/r,$$
    where~$M, \ell_0 \in \mathbb N$.
    Pick an arbitrary $m\in\mathbb N$ with $m\ge M$.
    Let~$G$ be an $r$-edge-coloured $n$-vertex complete graph with $n\ge(1+\eps)\cdot m\cdot t_r(H_i)$.

    We apply \cref{theorem:regularity} to~$G$ with parameters $r,\ell_0 \in \mathbb N$ and $0<\eps',d<1$ to obtain a partition $V_0,V_1, \dots,V_\ell$ with $\ell_0\le \ell\le M$ (recall we picked~$M$ so that $1/M\ll \eps',1/\ell_0$), a spanning subgraph~$G'$ of~$G$ and a reduced graph~$R$.
    Note that $|V_0|\le\eps' |G|$.
    As $\eps' \ll d$, \cref{reducedgraph} implies that  $\delta(R)\ge(1-2rd)|R|$.
    
    By \cref{theorem:fractional_bipartite} (recall $|R|\ge\ell_0\gg 1/d,r,|H_1|,\dots,|H_r|$) we have 
    $$\nu_{H_i}(R_i)=(1-d2^{3r})\cdot\frac{|H_i|}{t_r(H_i)}\cdot|R|,$$ 
    for some $i\in[r]$.
    By \cref{corollary:embedding}, it follows that~$G$ contains a monochromatic $H_i$-tiling in colour $i$ whose vertex set has size at least
    $$(1-\sqrt{\eps'})\cdot\nu_{H_i}(R_i)\cdot\frac{n-|V_0|}{|R|}\ge (1-\sqrt{\eps'})\cdot(1-d2^{3r})\cdot(1-\eps')\cdot\frac{n\cdot|H_i|}{t_r(H_i)}\ge m|H_i|,$$
    where the last inequality holds since $n\ge(1+\eps)\cdot m\cdot t_r(H_i)$ and $\eps\gg\eps',d$.
    In particular, $G$ contains a monochromatic copy of~$mH_i$ in colour~$i$. 
    It follows that $R_r(mH_i)\le(1+\eps)\cdot m\cdot t_r(H_i)$ provided~$m$ is sufficiently large.
\end{proof}

\subsection{Proof of Theorem~\ref{theorem:fractional_bipartite}}
The proof of \cref{theorem:fractional_bipartite} relies on the following structural lemma for edge-contractible graphs.

\begin{lemma}\label{lemma:key}
    Let~$G$ and~$H$ be graphs such that~$H$ is edge-contractible.
    There exist disjoint sets~$X,Y\subseteq V(G)$ such that
    $$\nu_H(G)=|Y|+\frac{|H|}{\sigma(H)}\cdot|X|$$ 
    and there are no edges in~$V(G)\setminus(X\cup Y)$, nor between~$Y$ and~$V(G)\setminus(X\cup Y)$.
\end{lemma}

\begin{proof}
    Pick a fractional $H$-tiling $w$ in~$G$ such that $w(G)=\nu_H(G)$. 
    Since~$H$ is edge-contractible, we may assume that if $w(\phi)>0$ then the the image $\phi$ of consists of precisely~$2$ elements, that is, $\phi$ maps $H$ to an edge in $G$.

    Let $X$ be the set of vertices~$v$ in~$G$ such that $w(v)=1$ and there exists a vertex~$u$ with $uv\in E(G)$ and $w(u)<1$.
    Let~$Y$ be the set of vertices~$v$ in~$G$ such that $w(v)=1$ and if $uv\in E(G)$ then $w(u)=1$.
    We assume~$w$ is chosen so that~$|X|+|Y|$ is minimal. 
    Pick~$\eps>0$ sufficiently small.

    Firstly, the set $V(G)\setminus(X\cup Y)$ is independent.
    Indeed, if there is an edge $xy\in V(G)\setminus(X\cup Y)$ then $w(x),w(y)<1$ and so increasing the weight of any homomorphism from~$H$ to~$xy$ increases the total weight of the fractional $H$-tiling, a contradiction.
    Furthermore, there are no edges between~$Y$ and~$V(G)\setminus(X\cup Y)$ by the definition of~$Y$.
    Thus, all edges incident to~$V(G)\setminus(X\cup Y)$ must be incident to~$X$.

    Since any homomorphism~$\phi\in\text{Hom}(G:H)$ with, say, $\text{Im}(\phi)=\{x,y\}$ satisfies $\sigma(H)\le\phi^{-1}(x),\phi^{-1}(x)\le|H|-\sigma(H)$, by the previous observation we conclude that
    $$\sum_{v\in V(G)\setminus(X\cup Y)}w(v)\le\frac{|H|-\sigma(H)}{\sigma(H)}\cdot|X|.$$
    It follows that 
    $$\nu_H(G)\le|X|+|Y|+\frac{|H|-\sigma(H)}{\sigma(H)}\cdot|X|=|Y|+\frac{|H|}{\sigma(H)}\cdot|X|.$$
    It remains to prove a corresponding lower bound.

    Suppose there exists $\phi_0\in\text{Hom}(G:H)$ such that $w(\phi_0)>0$ and the image $\{x,y\}$ of~$\phi_0$ intersects~$X$, say $x\in X$.
    If~$y\in X\cup Y$ then pick $z\in V(G)\setminus(X\cup Y)$ with $xz\in E(G)$ and~$w(z)<1$ (such~$z$ exists by the definition of~$X$).
    Let 
    $$\phi_1(v):=
    \begin{cases}
        x \text{ if $\phi_0(v)=x$}\\
        z \text{ if $\phi_0(v)=y$}
    \end{cases}$$
    and
    $$w'(\phi):=
    \begin{cases}
        w(\phi) &\text{ if $\phi\neq\phi_0,\phi_1$}\\
        w(\phi)-\eps &\text{ if $\phi=\phi_0$}\\
        w(\phi)+\eps &\text{ if $\phi=\phi_1$}.
    \end{cases}$$
    Then~$w'$ is a fractional $H$-tiling of~$G$ (for~$\eps$ sufficiently small).
    We have $w'(G)=w(G)=\nu_H(G)$.
    Furthermore, the set of vertices in $G$ with weight $1$ with respect to $w'$ is precisely $X\cup Y\setminus\{y\}$ (for $\eps$ sufficiently small).
    This contradicts the minimality of~$|X|+|Y|$.

    Hence, we have $y\in V(G)\setminus(X\cup Y)$.
    If $|\phi_0^{-1}(x)|>\sigma(H)$ then take a bipartition of~$H$ with parts consisting of~$\sigma(H)$ and~$|H|-\sigma(H)$ vertices respectively, and let~$\phi_2$ be the homomorphism sending the $\sigma(H)$-vertex class  of~$H$ to~$x$ and the $(|H|-\sigma(H))$-vertex class of~$H$ to~$y$.
    Let
    $$w''(\phi):=
    \begin{cases}
        w(\phi) &\text{ if $\phi\neq\phi_0,\phi_2$}\\
        w(\phi)-\eps &\text{ if $\phi=\phi_0$}\\
        w(\phi)+\eps &\text{ if $\phi=\phi_2$}.
    \end{cases}$$
    Then~$w''$ is a fractional $H$-tiling of~$G$ (for~$\eps$ sufficiently small).
    We have $w''(G)=w(G)=\nu_H(G)$.
    Furthermore, the set of vertices in $G$ with weight $1$ with respect to $w''$ is precisely $X\cup Y\setminus\{x\}$ (for $\eps$ sufficiently small).
    Again, this contradicts the minimality of~$|X|+|Y|$.

    \medskip

    We conclude that $y\in V(G)\setminus(X\cup Y)$ and~$|w_0^{-1}(x)|=\sigma(H)$.
    This holds for any homomorphism whose image intersects $X$, that is, such homomorphism sends $\sigma(H)$ vertices to $X$ and $|H|-\sigma(H)$ vertices to $V(G)\setminus(X\cup Y)$.
    Thus, we have
    $$\nu_H(G)=w(G)\ge |Y|+|X|+\frac{|H|-\sigma(H)}{\sigma(H)}\cdot|X|=|Y|+\frac{|H|}{\sigma(H)}\cdot|X|$$ 
    and so equality holds, as required. 
\end{proof}

\begin{proof}[Proof of Theorem~\ref{theorem:fractional_bipartite}]
    Let~$R$ be an $r$-edge-coloured graph with $\delta(R)\ge(1-2r\eps)|R|$.
    By \cref{lemma:key}, there are sets $X_1,\dots,X_r,Y_1,\dots,Y_r\subseteq V(G)$ such that for every~$i\in[r]$
    $$\nu_{H_i}(R_i)=|Y_i|+\frac{|H_i|}{\sigma(H_i)}\cdot|X_i|$$ 
    and there are no edges coloured with~$i$ in $V(G)\setminus(X_i\cup Y_i)$, nor between~$Y_i$ and~$X_i\cup Y_i$. 
    Therefore, every edge in~$G$ coloured with~$i$ is either incident to~$X_i$ or lies completely in~$Y_i$.

    Let $X:=X_1\cup\dots\cup X_r$ and $Z:=V(R)\setminus(X\cup Y_1\cup\dots\cup Y_r)$.
    For any $I,I'\subseteq[r]$ with $I\cap I'=\emptyset$, observe that all pairs of distinct vertices~$(x,y)$ with
    $$x\in Z\cup\left(\bigcap_{i\in I}(Y_i\setminus X)\right)\setminus \bigcup_{i\notin I}Y_i\quad\text{ and }\quad y\in Z\cup\left(\bigcap_{i\in I'}(Y_i\setminus X)\right)\setminus \bigcup_{i\notin I'}Y_i$$
    are not adjacent: the pair~$(x,y)$ does not intersect~$X$ and does not lie in any~$Y_i$.
    Since $\delta(R)\ge(1-2r\eps)|R|$, it follows that one of these two sets has size at most~$2r\eps|R|$.
    Select such set.

    \smallskip

    Let~$R'$ denote the graph obtained by deleting all selected vertex sets from~$R$.
    Since there are at most~$2^{2r}$ choices for~$I$ and~$I'$, we deleted at most~$(\eps\cdot2^{3r})|R|$ vertices.
    Let~$Y_i':=V(R')\cap Y_i\setminus X$ for every $i\in[r]$.
    Since all vertices in~$Z$ have been removed from~$R$, we conclude that every vertex in~$R'$ lies in some~$X_i$ or~$Y'_i$.
    Furthermore, if $I,I'\subseteq[r]$ are disjoint then
    $$\bigcap_{i\in I}Y'_i\setminus\bigcup_{i\notin I}Y'_i=\emptyset\quad\text{ or }\quad\bigcap_{i\in I'}Y'_i\setminus\bigcup_{i\notin I'}Y'_i=\emptyset$$
    by the definition of~$R'$.
    Let~$c_i:=|X_i|$ for every $i\in[r]$ and 
    $c_I:=|\bigcap_{i\in I} Y'_i\setminus(\bigcup_{i\notin I} Y'_i)|$
    for every $I\subseteq [r]$.
    Hence we have $|R'|=||{\bf c}||$ and
    $$\frac{|H_i|}{\sigma(H_i)}\cdot c_i+\sum_{i\in I\subseteq[r]}c_I=\frac{|H_i|}{\sigma(H_i)}\cdot|X_i| +|Y_i'|$$
    for every $i\in[r]$.
    Furthermore, if $I,I'\subseteq[r]$ are disjoint then $c_I=0$ or $c_{I'}=0$.
    
    By writing~$|R'|$ as
    $$|R'|=t_r\left(\frac{|R'|}{t_r(H_i)}H_i\right),$$
    it follows from \cref{define:lexico_cliquecover_parameter} that
    $$|Y_i'|+\frac{|H_i|}{\sigma(H_i)}\cdot|X_i|\ge\frac{|R'|\cdot|H_i|}{t_r(H_i)}$$
    for some $i\in[r]$.
    We have
    $$\nu_{H_i}(R_i)=\frac{|H_i|}{\sigma(H_i)}\cdot |X_i|+|Y_i|\ge\frac{|H_i|}{\sigma(H_i)}\cdot |X_i|+|Y_i'|\ge\frac{(1-\eps2^{3r})\cdot|R|\cdot|H_i|}{t_r(H_i)},$$
    as required.
\end{proof}

\section{Robust versions of  Theorems~\ref{theorem:r=3} and~\ref{theorem:chi>=r}}\label{sec:robust}
In this section we prove two results about monochromatic $H$-tilings in extremely dense (but non-complete) graphs; these results are then applied in the proof of
Theorem~\ref{randomramsey2} in Section~\ref{sec:random}.

\begin{lemma}\label{robust_chi}
Let $r\geq 2$ and~$H$ be a graph with $\chi(H)\ge r$.
Given any $\eta >0$, there exist $\delta >0$ such that the following holds. 
Provided $N>0$ is sufficiently large, if~$G$ is an $r$-edge-coloured graph on~$n\le N$ vertices with $e(G) \geq \binom{n}{2}-\frac{\delta N^2}{2}$, then~$G$ contains a monochromatic $H$-tiling whose vertex set has size at least 
$$\frac{|H|n}{\max\{||{\bf c}||:(F,{\bf c}) \text{ is an $(H,r)$-gadget}\}}-\eta N.$$ 
\end{lemma}

The proof is very similar to that of \cref{theorem:chi>=r'}, with appropriate adjustments.

\begin{proof}[Proof of \cref{robust_chi}]
    For brevity, we write $C:=\max\{||{\bf c}||:(F,{\bf c}) \text{ is an $(H,r)$-gadget}\}$.
    Note that $C\ge (r-1)|H|$ by \cref{remark:chi>=r}.

    By \cref{lemma:gadget_induction}, for each $(H,r)$-gadget $F$ we can select a capacity~${\bf q_F}$ of rational entries such that $(F,{\bf q_F})$ is an $(H,r)$-gadget.
   Let $m_F\in\mathbb N$ be the integer output by \cref{proposition:gadget_induction} on input $(F,{\bf q_F})$; set $m_0:=\max\{m_F:F\text{ is an $(H,r)$-gadget}\}$.
    
    Define constants
    \begin{align}\label{hier11}
    0<1/N\ll \delta \ll 1/\ell \ll  1/t
    \ll 1/m_0, 1/r,1/|H|,\eta,
    \end{align}
    where $\ell,t \in \mathbb N$ and~$\ell$ is divisible by $|H|$.
    
    We proceed by induction on~$n$.
    Since $C\ge|H|$, the base cases when $n\le\eta N$ are trivially true as we can simply take an empty $H$-tiling.
    Thus, suppose $n\ge\eta N$ and let~$G$ be an $r$-edge-coloured $n$-vertex  graph with $e(G)\ge\binom{n}{2}-\frac{\delta N^2}{2}$.

    Since $R_r(K_\ell)\le r^{r\ell}$,
    by repeatedly applying Tur\'an's theorem
    one obtains 
     a $K_\ell$-tiling $\mathcal K$ in~$G$ consisting of monochromatic copies of $K_\ell$ which cover all but at most ${\delta} ^{1/3} n$ vertices of~$G$.
    Indeed, if there are $s\ge {\delta} ^{1/3} n$ vertices uncovered
   then they span a subgraph of $G$ with at least $\binom{s}{2}-\frac{\delta N^2}{2}\ge(1-1/r^{r\ell})s^2/2$ edges; 
    the last inequality follows as $n \geq \eta N$ and by (\ref{hier11}).
    Hence this subgraph contains a copy of~$K_{r^{r\ell}}$ and so a monochromatic copy of~$K_\ell$.

    Suppose that, for some integer $g\ge {\delta} ^{1/3} n$, the collection~$\mathcal K$ includes~$g$ cliques in each colour.
    There are $g^r$ possible ways of picking a copy of~$K_\ell$ in each colour from this collection.
    If, for every such choice, the union of the picked cliques spans an induced subgraph of $G$ with a missing edge, then we have
    $$\delta ^{2/3} n^2\le\frac{g^r}{g^{r-2}}\le |E(\bar G)|\le\frac{\delta N^2}{2}
    \leq \frac{\delta n^2}{2 \eta ^2} ,$$
    where we used the fact that a missing edge is counted at most $g^{r-2}$ times.   
    This is a contradiction by (\ref{hier11}), hence we can pick a copy of~$K_\ell$ in each colour such that the subgraph spanned by the union of these cliques is complete.      
    Say the vertex set of the clique in colour $i$ is  $V_i$ for each $i \in [r]$.   
    By the bipartite version of Ramsey's theorem, since $\ell\gg t$, we can find a set $V_i'\subseteq V_i$ for each $i\in[r]$ such that
    \begin{itemize}
        \item $G[V_i',V_j']$ is a monochromatic complete bipartite graph for every $i,j\in[r]$ with $i\neq j$ and 
        \item $|V_i'|=t$ for every $i\in[r]$.
    \end{itemize}
    Let $B:=G[V_1'\cup\dots\cup V_r']$.
    Since $\chi(H)\ge r$ and $B[V_i']$ is coloured with~$i$ for every $i\in[r]$, by \cref{prop:mean_value} $B$ is a blow-up of an $(H,r)$-gadget $F$.
    As $t \geq m_0 \cdot |H|$,
    \cref{proposition:gadget_induction} implies  that~$B$ contains a blow-up~$B'$ of~$F$ such that $|B'|=m_F\cdot||{\bf q_F}||\le m_F\cdot C$ and~$B'$ contains a monochromatic copy of~$m_FH$ in each colour. 

    Let~$G':=G\setminus V(B')$.
    We have $|G'|=n-|B'|$ and $e(G')\geq\binom{n-|B'|}{2}-\frac{\delta N^2}{2}$.
    By the inductive hypothesis, $G'$ contains a monochromatic $H$-tiling covering at least $\frac{|H|(n-|B'|)}{C}-\eta N$ vertices.
    Since~$B'$ contains a monochromatic copy of~$m_FH$ in the same colour, we conclude there is a monochromatic $H$-tiling in~$G$ covering at least
    $$\left(\frac{|H|(n-|B'|)}{C}-\eta N\right)+m_F|H|\ge\left(\frac{|H|(n-|B'|)}{C}-\eta N\right)+\frac{|B'||H|}{C}=\frac{|H|n}{C}-\eta N$$
    vertices in~$G$, as required.
 
    Thus, we may now assume that~$\mathcal K$ contains at most ${\delta} ^{1/3}n$ copies of~$K_\ell$ in some colour.
    By the pigeonhole principle, $\mathcal K$ contains at least $\frac{|\mathcal K|-{\delta}^{1/3}n}{(r-1)}$ copies of~$K_\ell$ in the same colour.
    Since~$\ell$ is divisible by~$|H|$, each copy of~$K_\ell$ can be partitioned into $\ell/|H|$ copies of~$H$.
    Hence, there exists a monochromatic $H$-tiling in $G$ whose vertex set has size at least
    $$\ell\cdot\frac{|\mathcal K|-{\delta}^{1/3}n}
{(r-1)}
\ge\frac{(n-{\delta}^{1/3} n)-\ell{\delta}^{1/3}n}{(r-1)}
\stackrel{(\ref{hier11})}{\ge} \frac{n}{r-1}-\eta n\ge\frac{|H|n}{C}-\eta N,$$
    where the last inequality follows from $C\ge |H|(r-1)$ and $n\le N$.
\end{proof}

\begin{lemma}\label{robust_3}
Let~$H$ be a graph.
Given any $\eta >0$, there exist $\delta >0$ such that the following holds.
Provided $N>0$ is sufficiently large, if~$G$ is a $3$-edge-coloured graph on~$n\le N$ vertices with $e(G) \geq \binom{n}{2}-\frac{\delta N^2}{2}$, then~$G$ contains a monochromatic $H$-tiling whose vertex set has size at least 
$$\frac{|H|n}{\max\{||{\bf c}||:(F,{\bf c}) \text{ is an $(H,3)$-gadget}\}}-\eta N.$$ 
\end{lemma}

As before, we closely follow the proof of \cref{theorem:r=3'}.

\begin{proof}[Proof of \cref{robust_3}]
    For brevity, we write $C:=\max\{||{\bf c}||:(F,{\bf c}) \text{ is an $(H,r)$-gadget}\}$.
    If $H$ is non-bipartite, i.e., $\chi(H)\ge 3$, we are done by 
    \cref{robust_chi}.
    Therefore, suppose $H$ is bipartite.
    Note that \cref{remark:r=3} implies $C\ge 3|H|/2$. 

    \smallskip

    By \cref{lemma:gadget_induction}, for each $(H,3)$-gadget $F$ we can select a capacity ${\bf q_F}$ of rational entries such that $(F,{\bf q_F})$ is an $(H,3)$-gadget.
    Let $m_F\in\mathbb N$ be the parameter as in \cref{proposition:gadget_induction} with respect to $(F,{\bf q_F})$ and set $m_0:=\max\{m_F:F\text{ is an $(H,3)$-gadget}\}$.
    Pick constants 
    $$1/N\ll\delta\ll1/g\ll1/M\ll\eps\ll1/\ell_0\ll d\ll1/q_1 \ll 1/q \ll1/t\ll 1/|H|,1/m_0,\eta$$ 
    where~$N,g,M,\ell_0,q_1, q, t \in \mathbb N$ are positive integers, and~$q$ and~$t$ are divisible by $|H|$.
    
    We proceed by induction on~$n$.
    Since $C\ge|H|$, the base case $n\le\eta N$ is trivially true as we can simply take an empty $H$-tiling.
    Thus, suppose $n\ge\eta N$ and let~$G$ be a $3$-edge-coloured $n$-vertex complete graph with $e(G)\ge\binom{n}{2}-\frac{\delta N^2}{2}$.    
    For the inductive step, we will use the following claim.
    


    \begin{claim}\label{claim:inductive_step_robust}
    If there exists a subgraph~$B$ of~$G$ such that, for some~$m'\in\mathbb N$,
    \begin{itemize}
        \item $B$ contains a monochromatic copy of~$m'H$ in each colour and
        \item $|B|\le C\cdot m'$,
    \end{itemize}
    then~$G$ contains a monochromatic $H$-tiling whose vertex set has size at least $|H|n/C-\eta N$.  
    In particular, the second bullet point holds if $|B|\le 3|H|m'/2$.
    \end{claim}

    \begin{proofclaim}
        Let~$G':=G\setminus V(B)$.
        We have $|G'|=n-|B|$ and $e(G')\geq\binom{n-|B|}{2}-\frac{\delta N^2}{2}$.
        By the inductive hypothesis, $G'$ contains a monochromatic $H$-tiling covering
        $|H|(n-|B|)/C-\eta N$ vertices.
        Since~$B$ contains a monochromatic copy of~$m'H$ in the same colour, we conclude there is a monochromatic $H$-tiling in~$G$ whose vertex set has size at least
        $$\left(\frac{|H|(n-|B|)}{C}-\eta N\right)+m'|H|\ge\left(\frac{|H|(n-|B|)}{C}-\eta N\right)+\frac{|B||H|}{C}=\frac{|H|n}{C}-\eta N.$$
    \end{proofclaim}

    The next claim states that if~$G$ contains a (sufficiently large) blow-up of a rainbow triangle, then we can find a suitable subgraph~$B$ to apply Claim~\ref{claim:inductive_step_robust} to.
    The proof is word-by-word the same as that of Claim~\ref{claim:rainbow_triangle} in the proof of \cref{theorem:r=3'}, so we omit it.
    
    \begin{claim}
        If $G$ contains a blow-up~$B$ of a rainbow triangle with parts of size~$q$, then~$G$ contains a monochromatic $H$-tiling whose vertex set has size at least $|H|n/C-\eta N$.
    \end{claim}


    Similarly, the next claim states that if~$G$ contains a (sufficiently large) blow-up of a $3$-edge coloured $K_4$ containing a spanning rainbow star, then we can find a suitable subgraph~$B$ to apply Claim~\ref{claim:inductive_step_robust} to.
    We omit the proof, as it is again word-by-word identical to the proof of Claim \ref{claimK}.

    \begin{claim}\label{claimK_robust}
      Let $K$ be a $3$-edge-coloured  $K_4$ 
        such that $K$ contains a spanning rainbow star. 
        If $G$ contains a  blow-up~$B$ of $K$ with parts of size~$q$ spanning a complete subgraph in~$G$, then~$G$ contains a monochromatic $H$-tiling whose vertex set has size at least $|H|n/C-\eta N$.
    \end{claim}

    By the previous two claims, we may assume that~$G$ does not contain a blow-up of a rainbow triangle with parts of size~$q$, nor a blow-up of a $K_4$ as in   Claim~\ref{claimK_robust}.  

    \medskip

    As in the proof of \cref{theorem:r=3'}, we will use this structural information to find our desired monochromatic  $H$-tiling, namely by applying Lemma~\ref{theorem:regularity} and considering the structure of the reduced graph.
    There are two notable adjustments compared to the proof of \cref{theorem:r=3'}:

    \begin{itemize}
        \item We apply Lemma~\ref{theorem:regularity} to an almost-spanning subgraph $G_0$ of $G$ with large minimum degree, rather than $G$. 
        The large minimum degree of $G_0$ ensures the reduced graph also has large minimum degree, and the set $V(G)\setminus V(G_0)$ is sufficiently small.
        
        \item As in the proof of \cref{theorem:r=3'}, we will argue that the reduced graph contains no rainbow triangles nor $3$-edge coloured $K_4$ as in Claim~\ref{claimK_robust}.
        The latter case is more delicate: if the reduced graph contains such a $K_4$, then $G$ contains a blow-up this $K_4$, but this might not span a complete subgraph  and so we cannot immediately reach a contradiction.
        Thus, further steps are required.
    \end{itemize}

    Let~$G_0$ be the induced subgraph of~$G$ obtained by removing the vertices incident to at least $2\delta n/(\eps\eta^2)$ missing edges.
    Note that if we have removed more than $\eps n/2$ vertices, then
    $$\frac{1}{2}\cdot\frac{\eps n}{2}\cdot\frac{2\delta n}{\eps\eta^2}<E(\bar G)\le \frac{\delta N^2}{2}\le\frac{\delta n^2}{2\eta^2 },$$
    a contradiction.
    Thus, $|G_0|\ge(1-\eps/2)n$ and $$\delta(G_0)\ge|G_0|-\frac{2\delta n}{\eps\eta^2}\ge\left(1-\frac{4\delta }{\eps\eta^2}\right)|G_0|\ge(1-\eps)|G_0|,$$
    where the last inequality follows since $\delta\ll\eps, \eta$.

    \medskip

    We apply \cref{theorem:regularity} to~$G_0$ with parameters $r:=3,\ell_0 \in \mathbb N$ and $0<\eps/2,d<1$ to obtain a partition $V_0,V_1, \dots,V_\ell$ with $\ell_0\le \ell\le M$ (recall we picked~$M$ so that $1/M\ll \eps,1/\ell_0$ and $|G_0|\geq n/2\geq \eta N/2\geq M$), a spanning subgraph~$G'$ of~$G_0$ and a reduced graph~$R$.
    Note that $|V_0|\le(\eps/2)|G_0|\leq(\eps/2)n$ and so $|V_i|=(|G_0|-|V_0|)/\ell\ge(1-\eps)n/M$ for every $i\in[\ell]$.
    As $\eps \ll d$, \cref{reducedgraph} implies that  $\delta(R)\ge(1-6d)|R|$.

    Recall that~$R$ is $3$-edge-coloured and has vertex set $\{V_1,\dots,V_\ell\}$. Further,   if $V_xV_y\in E(R)$ is coloured $c \in [3]$, then $G'_c[V_x,V_y]$ is $[\eps/2,d]$-regular (and thus also $[\eps,d]$-regular).

    \begin{claim}\label{claimR_robust}
    $R$ contains no rainbow star on $4$ vertices nor a rainbow triangle.
    \end{claim}
    
    \begin{proofclaim}
    If $R$ contains a rainbow triangle, then  by \cref{corollary:embedding}
    $G' \subseteq G_0\subseteq G$ contains a blow-up of a rainbow triangle with parts of size $q$, a contradiction to our assumption.
 
    Suppose~$R$ contains a rainbow star on $4$ vertices, say with center $V_c$ and leaves $V_1$, $V_2$ and~$V_3$, where $V_cV_i$ is coloured with $i$ and so $G_i'[V_c,V_i]$ is $[\eps,d]$-regular, for every $i\in[3]$.
    Take a random subset $V_i'\subseteq V_i$ of size~$g$, for every $i\in\{c\}\cup [3]$.
    The probability a fixed pair of vertices lies in the subgraph $G[V_c'\cup V_1'\cup V_2'\cup V_3']$ is at most
    $g^2/|V_c|^2$.
    Hence, the expected number of missing edges in~$G[V_c'\cup V_1'\cup V_2'\cup V_3']$ is at most
    $$
    \frac{g^2}{|V_c|^2}\cdot\frac{2\delta n}{\eps\eta^2}\cdot 4|V_c|
    \le \frac{8g^2\delta M}{\eps\eta^2 (1-\eps)}
    \le \frac{1}{10},
    $$
    where the first inequality follows from $|V_c|\ge(1-\eps)n/M$ and the second inequality follows from $\delta\ll 1/M,1/g,\eta,\eps$.
    By Markov's inequality, the probability~$G[V_c'\cup V_1'\cup V_2'\cup V_3']$ has a missing edge (i.e, it is not a complete subgraph of $G$) is at most $1/10$.

    Furthermore, since the pairs $G_1'[V_c,V_1]$, $G_2'[V_c,V_2]$ and $G_3'[V_c,V_3]$ are $[\eps,d]$-regular, it follows that the set $T$ of vertices $v\in V_c$ such that $d_{G_i'}(v,V_i')\ge d|V_c|/2$ for every $i\in[3]$ has size $|T|\geq(1-3\eps)|V_c|$.
    Indeed, if that is not the case, then there exists a set $T_0\subseteq V_c$ of at least $\eps|V_c|$ vertices and some $i\in[3]$ such that each vertex $v\in T_0$ satisfies $d_{G_i'}(v,V_i)\le d|V_c|/2$.
    In particular, $d_{G'_i}(T_0,V_i)\le d/2$, which contradicts the fact that $G_i'[V_c,V_i]$ is $[\eps,d]$-regular since $\eps\ll d$.
    
    If $v\in V_c$ satisfies $d_{G_i'}(v,V_i)\ge d|V_c|/2$ for some $i\in[3]$, then by Chernoff's bound we have
    $$\mathbb P(d_{G'}(v,V_i')\le dg/4)\le\exp(-dg/24).$$
    Let $P$ be the set of vertices $v\in V_c'$ such that $d_{G_i'}(v,V_i')\leq dg/4$ for some $i\in[3]$.
    Then
    \begin{align*}
    \mathbb E(|P|)&=\sum_{v\in V_c}\mathbb P(v\in P)\\
    &\le\sum_{v\in V_c\setminus T}\mathbb P(v\in V_c')+\sum_{v\in T}\mathbb P(v\in V_c')\cdot\mathbb P(d_{G'_i}(v,V_i')\le dg/4 \text{ for some $i\in[3]$})\\
    &\le(|V_c|-|T|)\cdot\frac{g}{|V_c|}+|T|\cdot\frac{g}{|V_c|}\cdot3\exp(-dg/24)\\
    &\leq 3\eps g+(1-3\eps)g\cdot3\exp(-dg/24)\le g/100,
    \end{align*}
    where the last inequality holds since $1/g\ll d$.
    By Markov's inequality,
    $$\mathbb P(|P|\ge g/2)\le \mathbb E(|P|)\cdot\frac{2}{g}\le\frac{1}{50}.$$
    Hence, there exists a choice of $V_c'$, $V_1'$, $V_2'$, $V_3'$ such that together they span a complete graph in $G$ and at least $g/2$ vertices $v\in V_c'$ satisfy $d_{G'_i}(v,V_i')\ge dg/4$ for every $i\in[3]$.
    Now, pick random sets~$S_i\subseteq V_i'$ of size~$q_1$ for every $i\in[3]$.
    Let $S_0$ be the set of vertices $v\in V_c'$ such that $v$ is a neighbour of all vertices in $S_i$ in $G'_i$, for every $i\in[3]$.
    The probability a vertex $v\in V_c'$ is in $S_0$ is precisely
    $$\binom{d_{G'_1}(v,V'_1)}{q_1}\cdot\binom{d_{G'_2}(v,V'_2)}{q_1}\cdot\binom{d_{G'_3}(v,V'_3)}{q_1}\cdot\binom{g}{q_1}^{-3}$$
    and so we have
    $$\mathbb E(|S_0|)=\sum_{v\in V_c'}\binom{d_{G'_1}(v,V'_1)}{q_1}\cdot\binom{d_{G'_2}(v,S_2)}{q_1}\cdot\binom{d_{G'_3}(v,V'_3)}{q_1}\cdot\binom{g}{q_1}^{-3}.$$
    As at least $g/2$ vertices $v\in V_c'$ satisfy $d_{G'_i}(v,V_i')\ge dg/4$ for every $i\in[3]$, we have
    $$\mathbb E(|S_0|)\geq \frac{g}{2}\cdot\binom{dg/4}{q_1}^3\cdot\binom{g}{q_1}^{-3}\ge\frac{g}{4}\left(\frac{d}{4}\right)^{3q_1}\ge q_1,$$
    where the last inequality holds since $1/g\ll d, 1/q_1$.
    Hence, there exist vertex-disjoint sets $S_0$, $S_1$, $S_2$ and~$S_3$, each of size~$q_1$, such that $G[S_0\cup S_1\cup S_2\cup S_3]$ is complete and $G'_i[S_0, S_i]$ is complete bipartite for every $i\in[3]$.
    As $1/q_1 \ll 1/q$ and $G[S_0\cup S_1\cup S_2\cup S_3]$ is a complete graph, the bipartite version of Ramsey's theorem implies that we can find subsets $S'_i \subseteq S_i$ of size $q$ for each $i \in \{0\} \cup [3]$
    such that $G[S'_1,S'_2]$, $G[S'_1,S'_3]$ and $G[S'_2,S'_3]$ are each monochromatic bipartite graphs. 
    Hence $G[S'_0\cup S'_1 \cup S'_2 \cup S'_3]$ is a blow-up of a $K_4$ as in Claim~\ref{claimK_robust}.
    This is a contradiction.
    \end{proofclaim}

    By the previous claim, $R$ contains no rainbow  triangle nor a rainbow star on $4$ vertices.
    Following word-by-word the  proof of \cref{theorem:r=3'} after Claim~\ref{claimR}, we conclude that $G_0$ contains a monochromatic copy of~$mH$ where
    $m\ge\frac{|G_0|}{C}$.
    In particular, this copy of~$mH$ covers at least $(1-\eps)n|H|/C\ge n|H|/C-\eta N$ vertices of $G$, since $n\leq  N$ and $\eps\ll \eta, 1/|H|$.
    This concludes the proof.
\end{proof}

\section{Proof of Theorem~\ref{randomramsey2}}\label{sec:random}
Let $H$, $r$ and $\gamma$ be as in the statement of the theorem. Set $h:=|H|$. Define additional constants $\eta, \eps >0$ and $C , T, t\in \mathbb N$ such that
\begin{align}\label{hierrandom}
            0< \frac{1}{C} \ll \eta \ll \frac{1}{T} \ll \frac{1}{t}     \ll \eps \ll \gamma, \frac{1}{r},\frac{1}{h}.
    \end{align}

Let $p \geq Cn^{-1/m_2(H)}$ and consider $G:=G_{n,p}$. 
Note that $p \geq C/n$ as $m_2(H)\geq 1$; so since $1/C \ll \eta$, Chernoff's bound implies that a.a.s.\@  every disjoint pair $U,V \subseteq V(G)$ with $|U|,|V| \geq \eta n$ satisfies
\begin{align}\label{densitybound}
  1/2 \leq d_p (U,V)  \leq 2.
\end{align}
Thus, a.a.s.\@ $G$ is $(\eta, p,2)$-upper-uniform.
Note that $G$ also a.a.s.\@ satisfies the conclusion of Proposition~\ref{randomklr} on input $H$, $d:=1/5$ and $\eta$, and with $\sqrt{\eps}$ playing the role of $\eps$.

Consider any $r$-edge-colouring of $G$ using colours $1,\dots, r$. Let $G_1,\dots, G_r$ denote the spanning subgraphs of $G$ such that, for each $k \in [r]$, the edge set of $G_k$  consists precisely of those
edges in $G$ coloured $k$. Since $G$ is $(\eta, p,2)$-upper-uniform,
each of $G_1,\dots, G_r$ is $(\eta, p,2)$-upper-uniform. Thus,
the choice of $\eps, t,T, \eta$ ensures we can apply Lemma~\ref{sparserl} to obtain 
an equipartition $V_1,\dots, V_s$ of $V(G)$ into $s$ parts, where $t \leq s \leq T$, and for which all but at most an $\eps$-proportion of the pairs $(V_i,V_j)$ 
    ($i \neq j \in [s]$) induce an $(\eps, p)$-regular pair in each of the graphs $G_1,\dots, G_r$.

Let $R$ be the (reduced)  graph with vertex set $[s]$, where $ij$ is an edge precisely if $(V_i,V_j)$ 
     induce an $(\eps, p)$-regular pair in each of the graphs $G_1,\dots, G_r$.
Thus, we have that $e(R)\geq (1-\eps)\binom{s}{2}$.

Note that if $ij\in E(R)$, then as $|V_i|,|V_j| \geq \lfloor n/T\rfloor \stackrel{(\ref{hierrandom})}{\geq} \eta n$, (\ref{densitybound}) implies that
$d_p(V_i,V_j) \geq 1/2$. For each $k \in [r]$, write $d_{p,G_k}(V_i,V_j)$ for the value of  $d_{p}(V_i,V_j)$ in $G_k$.
Let $k_{ij}$ denote the choice of $k$ that maximises the value of $d_{p,G_k}(V_i,V_j)$; so
$d_{p,G_{k_{ij}}}(V_i,V_j) \geq 1/(2r)$.
Define an $r$-edge-colouring of $R$ by assigning each edge $ij\in E(R)$ colour $k_{ij} \in [r]$.

Suppose that $r=2$. As $\eps \ll \gamma$, we can choose $\eps$ to be much smaller than the output $\delta$ of Lemma~\ref{robust2colours} on input $\gamma/2$.
In particular, Lemma~\ref{robust2colours} implies that $R$ contains a
monochromatic $H$-tiling $\mathcal H_R$ covering at least
$$
\left(\frac{|H|}{2|H|-\alpha (H)}-\frac{\gamma}{2} \right )s
$$
of the vertices of $R$.
Without loss of generality assume that $\mathcal H_R$ has colour $1$.

\begin{claim}\label{claimrandom}
    Suppose that $i_1,\dots ,i_{h}$ are the vertices of a copy $H'$ of $H$ in $\mathcal H_R$. Then $G_1[V_{i_1}\cup \dots \cup V_{i_h}]$ contains an $H$-tiling covering
    all but at most $\gamma ^2 n/s$ vertices in each of $V_{i_1}, \dots , V_{i_h}$.
\end{claim}
\begin{proofclaim}
    Without loss of generality we may assume that $i_1=1, \dots , i_h=h$.
Consider any $W_{1} \subseteq V_{1}, \dots, W_{h} \subseteq V_{h}$
such that $n_1:=|W_{1}|=\dots =|W_{h}|\geq \gamma ^2 n/s $.
If $ij$ is an edge in $H'$ then $d_{p,G_1}(V_i,V_j)\geq 1/(2r)=1/4$
and $(V_i,V_j)$ induces an $(\eps,p)$-regular pair in $G_1$.
By Fact~\ref{fact1} and (\ref{hierrandom}), we have that
$(W_i,W_j)$ induces a $(\sqrt{\eps}, p)$-regular pair in $G_1$ containing at least
$$
(d_{p,G_1}(V_i,V_j)- \eps )p|W_i||W_j|\geq \frac{pn_1^2}{5}
$$
edges.

Therefore, $G_1[W_1\cup \dots \cup W_h]$ belongs to $\mathcal G (H,n_1,\textbf{m}  , p , \sqrt{\eps}) $ for some $\textbf{m}=(m_{ij} )_{ij \in E(H')}$ with
$m_{ij} \geq {pn_1^2}/{5}$ for all $ij \in E(H')$.
By Proposition~\ref{randomklr}, $G_1[W_1\cup \dots \cup W_h]$ contains 
a canonical copy of $H$; that is, a copy of $H$ with precisely one vertex in each 
class $W_1,\dots, W_h$.

Using this property we can now greedily obtain the desired $H$-tiling
in $G_1[V_1\cup \dots \cup V_h]$.
\end{proofclaim}

By applying Claim~\ref{claimrandom} to each copy of $H$ in $\mathcal H_R$, we
obtain a (monochromatic) $H$-tiling in $G_1$ covering at least
$$
\left(\frac{|H|}{2|H|-\alpha (H)}-\frac{\gamma}{2} \right )s \times \left (1-\gamma ^2 \right )\frac{n}{s} \geq \left(\frac{|H|}{2|H|-\alpha (H)}-{\gamma} \right )n
$$
vertices, thereby confirming the validity of ($\alpha _2$).

The proofs of ($\alpha _3$) and ($\alpha _4$) follow analogously by invoking (the $n=N$ cases of) \cref{robust_3} and \cref{robust_chi} respectively instead of Lemma~\ref{robust2colours}.

\smallskip

We next verify ($\alpha _1$). Since $e(R)\geq (1-\eps)\binom{s}{2}$, as in the proof of Lemma~\ref{robust2colours} we can apply 
Tur\'an's theorem and then Ramsey's theorem to find a monochromatic copy of $H$ in $R$.
Greedily repeating this process we obtain an $H$-tiling $\mathcal H_R$ in $R$ covering all but at most $\gamma  s/2$ vertices in $R$ so that each copy of $H$
in $\mathcal H_R$ is monochromatic. 

One can now argue as before to obtain an analogue of Claim~\ref{claimrandom} for each copy of $H$ in $\mathcal H_R$. We therefore obtain an $H$-tiling $\mathcal H_1$ in $G$ covering 
at least
$$
(1-\gamma /2)s \times \left (1-\gamma ^2 \right )\frac{n}{s} \geq (1-\gamma )n
$$
vertices where each copy of $H$ in $\mathcal H_1$ is monochromatic, as desired.\qed

\smallskip
{\noindent \bf Data availability statement.}
There are no additional data beyond that contained within the main manuscript.

{\noindent \bf Open access statement.}
	This research was funded in part by the  EPSRC grant   UKRI1117. For the purpose of open access, a CC BY public copyright licence is applied to any Author Accepted Manuscript arising from this submission.

\section*{Appendix}

\subsection*{Proof of \cref{thm:rambi}}

Firstly, we describe an $r$-edge-coloured complete graph that does not contain a monochromatic copy of~$mH$.
This construction has been observed in, e.g., \cite{GyarfasSS}.

\begin{lemma}\label{lemma:no_bipartite_component}
    Let~$r\ge3$ and~$m\ge1$ and
    let~$H$ be a graph with no bipartite components.
   Then  $R_r(mH)\ge rm|H|-r+1$.
\end{lemma}

\begin{proof}
    Let $V_1\cup\dots\cup V_r$ be a partition of a complete graph~$K$ into~$r$ parts where $|V_i|=m|H|-1$ for each~$i\in[r]$.
    For each  $i\in[r]$ in turn, we assign colour~$i$ to all edges in~$V_i$ and to all previously uncoloured edges between~$V_{i+1}$ and~$V(K)\setminus\{V_{i},V_{i+1}\}$.\footnote{The subscript of~$V_{i+1}$ is taken modulo~$r$.}
    As~$r\ge3$, each edge receives a colour.
    Furthermore, since~$H$ has no bipartite components, all monochromatic copies of~$H$ in colour~$i$ must lie completely in~$V_i$, for every $i\in[r]$.
    Since $|V_i|<m|H|$ for every $i\in[r]$, it follows that there is no monochromatic copy of~$m|H|$.
    Thus $R_r(mH)\ge|K|+1=rm|H|-r+1$, as required.
\end{proof}

We are ready to prove \cref{thm:rambi}.

\begin{proof}[Proof of \cref{thm:rambi}]
    We first prove the following claim.
    \begin{claim}
        If $R_r(mH)\le rm|H|+C$ for some $C\in\mathbb Z$ with $C\ge R_r(H)-(m+r)|H|$,
        then $R_r((m+1)H)\le r(m+1)|H|+C$.
    \end{claim}

    \begin{proofclaim}
        Let~$G$ be an $r$-edge-coloured graph on $r(m+1)|H|+C$ vertices.
        It suffices to show that~$G$ contains a monochromatic copy of~$(m+1)H$.
        First, we find a collection~$\mathcal H$ of vertex-disjoint monochromatic copies of~$H$ which cover all but at most~$R_r(H)-1$ vertices of~$G$. 

        Suppose that each colour appears in~$\mathcal H$.
        Let~$G'$ be the graph obtained by removing a copy of~$H$ in~$\mathcal H$ of colour~$i$ from~$G$, for every $i\in[r]$.
        Then $|G'|=|G|-r|H|= rm|H|+C$.
        By assumption, $G'$ contains a monochromatic copy of~$mH$, say in colour $j\in[r]$.
        Among the removed copies of~$H$ in~$\mathcal H$, one is monochromatic in colour~$j$.
        We thus conclude there is a monochromatic copy of $(m+1)H$ in~$G$, as required.

        Suppose instead that some colour does not appear in~$\mathcal H$.
        By the pigeonhole principle, $\mathcal H$ includes at least $(|G|-R_r(H)+1)/((r-1)|H|)$ copies of~$H$ of the same colour.
        Note that
        \begin{align*}
            \frac{|G|-R_r(H)+1}{(r-1)|H|}&=\frac{r(m+1)|H|+C-R_r(H)+1}{(r-1)|H|}\\
            &=\frac{(r-1)m|H|+(m+r)|H|+C-R_r(H)+1}{(r-1)|H|}\\
            &\ge\frac{(r-1)m|H|+1}{(r-1)|H|}>m.
        \end{align*}
        Hence, $G$ contains a monochromatic copy of~$(m+1)H$, as required.
    \end{proofclaim}

    For every integer $m\ge1$, let $C_m$ be so that $R_r(mH)=rm|H|+C_m$.
    We have $C_m\ge-r+1$ for every $m\ge1$ by \cref{lemma:no_bipartite_component}.
    In particular, $C_m\ge R_r(H)-(m+r)|H|$ provided~$m$ is sufficiently large.
    Therefore, by the above claim, the sequence $C_1,C_2,\dots$ is eventually non-increasing.
    In particular, the sequence $C_1,C_2,\dots$ is eventually constant; that is, there exists an integer~$C$ such that $R_r(mH)=rm|H|+C$ for any~$m$ sufficiently large. 

    By \cref{lemma:no_bipartite_component}, we have $C\ge -r+1$.
    Set $C':=R_r(H)-r|H|$.
    So  $R_r(H)=r|H|+C'$ and
     for every $m\ge 1$ we have $C'\ge R_r(H)-(r+m)|H|$.
    Hence, by the above claim, we have $R_r(mH)\le rm|H|+C'$ for every $m\ge1$.
    This implies $C\le C'= R_r(H)-r|H|$, as required.
\end{proof}

\end{document}